\documentclass[a4paper, 11pt]{amsart}
\usepackage{amsmath, amssymb, amsfonts, amsthm, enumerate, mathtools, tikz-cd, hyperref}
\usepackage{hyperref}
\usepackage[margin=1.2in]{geometry}
\usepackage[backend=biber]{biblatex}
\renewbibmacro{in:}{}
\usepackage{booktabs}
\usepackage{bm}
\usepackage{bbm}
\usepackage{dirtytalk}
\usepackage{tikz}
\usepackage{tikz-cd}
\usepackage{pgf}
\usepackage{graphicx}
\usepackage{wrapfig}
\usepackage{ragged2e}
\usepackage{mathabx}
\usepackage{mathrsfs}
\usepackage{multirow}
\usepackage{multicol}
\usepackage[T1]{fontenc}
\usepackage{stmaryrd}

\usepackage{color,soul}
\usepackage{mathtools}
\usepackage{amsmath}
\usepackage{amsthm}
\usepackage{tikz-cd}
\usepackage[T2A,T1]{fontenc}
\numberwithin{equation}{section}

\newtheorem{thm}{Theorem}[section]
\newtheorem{lem}[thm]{Lemma}
\newtheorem{defi}[thm]{Definition}
\newtheorem{rem}[thm]{Remark}
\newtheorem{rems}[thm]{Remarks}

\newtheorem{prop}[thm]{Proposition}

\newtheorem{assum}[thm]{Assumption}

\newtheorem{thmx}{Theorem}

\newcommand{\NN}{\mathbb{N}}	
\newcommand{\ZZ}{\mathbb{Z}}	
\newcommand{\QQ}{\mathbb{Q}}	
\newcommand{\BB}{\mathbb{B}}
\newcommand{\CC}{\mathbb{C}}
\newcommand{\RR}{\mathbb{R}}
\newcommand{\AAA}{\mathbb{A}}
\newcommand{\BR}{\mathbb{B}}

\newcommand{\HH}{\mathcal{H}}

\newcommand{\Dcris}{\mathbb{D}_{\mathrm{crys}}}

\newcommand{\Oe}{\mathcal{O}_{E}}

\newcommand{\GG}{\Gamma}

\newcommand{\OO}{\mathcal{O}}
\newcommand{\vp}{\varphi}
\newcommand{\lv}{\Vert}

\newcommand{\Qp}{\QQ_p}
\newcommand{\Cp}{\CC_p}
\newcommand{\Zp}{\ZZ_p}

\newcommand{\Brig}{\BB^{+}_{\mathrm{rig},\Qp}}
\newcommand{\BrigE}{\BB^{+}_{\mathrm{rig},E}}

\newcommand{\intiwG}{\Lambda_{\Oe}(\GG_1)}

\newcommand{\Mbar}{\underline{M}}

\newcommand{\LW}{\mathcal{L}_{W}}

\newcommand{\MatA}{\begin{pmatrix}
0 & -1/vp^{k-1} \\[6pt]
1 & a/vp^{k-1}
\end{pmatrix}}
\newcommand{\MatQ}{\begin{pmatrix}
\alpha & -\beta \\[6pt]
-vp^{k-1} & vp^{k-1}
\end{pmatrix}}
\newcommand{\Hiw}{\mathrm{H}^{1}_{\mathrm{Iw}}(\Qp,W)}

\newcommand{\LT}{\mathcal{L}_{T_W}}
\newcommand{\Matwach}{\begin{bmatrix}(1+\pi)\vp(n_1) & (1+\pi)\vp(n_2)\end{bmatrix}}

\newcommand{\HiwT}{\mathrm{H}^{1}_{\mathrm{Iw}}(\Qp,T_W)}
\newcommand{\OT}{\Omega_{T_{W},k-1}}
\newcommand{\OW}{\Omega_{W,k-1}}
\newcommand{\ideles}{\mathbb{A}_{K}^{\times}}
\newcommand{\PP}{\mathfrak{p}}
\newcommand{\PPP}{\overline{\mathfrak{p}}}

\newcommand{\qq}{\mathfrak{q}}
\newcommand{\pset}{\{\mathfrak{p},\overline{\mathfrak{p}}\}}
\newcommand{\qqq}{\mathfrak{q}}
\newcommand{\tilX}{\widetilde{\Xi}}
\newcommand{\FF}{\mathcal{F}}
\newcommand{\Mbarp}{\underline{M_{\mathfrak{p}}}}
\newcommand{\Mbarpp}{\underline{M_{\overline{\mathfrak{p}}}}}
\newcommand{\Laa}{L_{\alpha_{\PP},\alpha_{\PPP}}}
\newcommand{\Lab}{L_{\alpha_{\PP},\beta_{\PPP}}}
\newcommand{\Lba}{L_{\beta_{\PP},\alpha_{\PPP}}}
\newcommand{\Lbb}{L_{\beta_{\PP},\beta_{\PPP}}}
\newcommand{\Lsa}{L_{\sharp,\alpha_{\PPP}}}
\newcommand{\Lfa}{L_{\flat,\alpha_{\PPP}}}
\newcommand{\Lsb}{L_{\sharp,\beta_{\PPP}}}
\newcommand{\Lfb}{L_{\flat,\beta_{\PPP}}}
\newcommand{\delx}{\dfrac{\partial}{\partial X}}
\newcommand{\matqp}{Q^{-1}_{\PP}}
\newcommand{\delp}{\partial_{\PP}}
\newcommand{\Lss}{L_{\sharp,\sharp}}
\newcommand{\Lsf}{L_{\sharp,\flat}}
\newcommand{\Lfs}{L_{\flat,\sharp}}
\newcommand{\Lff}{L_{\flat,\flat}}

\newcommand{\Bstd}{\mathcal{B}}
\newcommand{\Beigen}{\mathcal{B}_{\mathrm{eig}}}

\newcommand{\gl}{\mathrm{GL}_{2}}

\newcommand{\frkn}{\mathfrak{n}}

\newcommand{\frkm}{\mathfrak{m}}

\DeclareMathOperator{\Art}{Art}
\DeclareMathOperator{\Gal}{Gal}
\DeclareMathOperator{\supt}{sup_{t}}

\DeclareMathOperator{\modd}{mod}
\DeclareMathOperator{\Col}{\underline{Col}}
\DeclareMathOperator{\adj}{adj}
\DeclareMathOperator{\rec}{rec}
\DeclareMathOperator{\mat}{Mat}

\usepackage{color,xcolor}
\newcommand{\green}[1]{\textcolor{green}{#1}}

\begin{document}
\baselineskip 18pt

\title{ Signed $p$-adic $L$-functions of Bianchi modular forms}
\author{Mihir V. Deo} 
\email{mdeo048@uottawa.ca}
\address{Department of Mathematics and Statistics, University of Ottawa, Ottawa, ON, Canada K1N 6N5}
\date{}


\begin{abstract}
Let $p\geq 3$ be a prime number and $K$ be a quadratic imaginary field in which $p$ splits as $\mathfrak{p}\overline{\mathfrak{p}}$. Let $\mathcal{F}$ be a cuspidal Bianchi eigenform over $K$ of weight $(k,k)$, where $k\geq 0$ is an integer, level $\mathfrak{m}$ coprime to $p$, and non-ordinary at both of the primes above $p$. We assume $\mathcal{F}$ has trivial nebentypus. For $\mathfrak{q}\in\{\mathfrak{p}, \overline{\mathfrak{p}}\}$, let $a_{\mathfrak{q}}$ be the $T_{\mathfrak{q}}$ Hecke eigenvalue of $\mathcal{F}$ and let $\alpha_{\mathfrak{q}},\beta_{\mathfrak{q}}$ be the roots of polynomial $X^{2} -a_{\mathfrak{q}}X+ p^{k+1}$. Then we have four $p$-stabilizations of $\mathcal{F}$: $\mathcal{F}^{\alpha_{\mathfrak{p}},\alpha_{\overline{\mathfrak{p}}}}, \mathcal{F}^{\alpha_{\mathfrak{p}},\beta_{\overline{\mathfrak{p}}}}, \mathcal{F}^{\beta_{\mathfrak{p}},\alpha_{\overline{\mathfrak{p}}}},$ and $ \mathcal{F}^{\beta_{\mathfrak{p}},\beta_{\overline{\mathfrak{p}}}}$ which are Bianchi cuspforms of level $p\mathfrak{m}$. By the works of Williams, to each $p$-stabilization $\mathcal{F}^{*,\dagger}$, we can attach a locally analytic distribution $L_{p}(\mathcal{F}^{*,\dagger})$ over the ray class group $\text{Cl}(K,p^{\infty})$. On viewing $L_{p}(\mathcal{F}^{*,\dagger})$ as a two-variable power series with coefficients in some $p$-adic field having unbounded denominators satisfying certain growth conditions, we decompose this power series into a linear combination of power series with bounded coefficients in the spirit of Pollack, Sprung, and Lei--Loeffler--Zerbes.
\end{abstract}
\maketitle
\section{Introduction}
The study and construction of $p$-adic $L$-functions of arithmetic objects, like modular forms, is one of the central topics in modern number theory. 
The analytic $p$-adic $L$-functions are  
 distributions on $p$-adic Lie groups like $\Zp$.
For example, let $f$ be an elliptic modular eigenform of weight $k\geq 2$, level $N$ and character $\epsilon$, and let $p$ be a prime such that $p\nmid N$. 
Let $\alpha$ be a root of the Hecke polynomial $X^2- a_{p}X+\epsilon(p)p^{k-1}$ such that $v_{p}(\alpha)<k-1$, where $v_{p}$ is the normalized $p$-adic valuation such that $v_{p}(p)=1$, and $a_{p}$ is 
the $T_{p}$-eigenvalue of $f$. 
Then, due to the constructions of Amice-Velu and Vishik (see \cite{AV}, \cite{Vishik}) we can attach to $f$ a $p$-adic distribution $L_{p}(f,\alpha)$ of order $v_{p}(\alpha)$ 
over $\Zp^\times$.
This $L_{p}(f,\alpha)$ interpolates the critical values of the complex $L$-function of $f$.

We continue with the example of $p$-adic $L$-functions  
of modular forms.
When $f$ is good ordinary at $p$, i.e. $v_{p}(a_p)=0$, 
$L_{p}(f,\alpha)$ is a bounded measure and hence an element of the Iwasawa algebra $\Lambda_{K}(\GG)\cong K\otimes \OO_{K}[\Delta][[\GG_{1}]]$, where $K$ is some finite extension of $\Qp$, $\GG=\Gal(\Qp(\mu_{p^\infty})/\Qp)\cong \Delta\times\GG_{1}$, $\Delta\cong (\ZZ/p\ZZ)^{\times}$, and $\GG_{1}\cong\Zp$.
In this setting, 
the arithmetic is well understood and we have 
an Iwasawa main conjecture 
which relates this $p$-adic $L$-function with the characteristic ideal of the Selmer group of $f$ (proved in many cases by Kato in \cite{Kato}, Skinner-Urban in \cite{SU}, Wan in \cite{Wan}, etc).

When $f$ is good non-ordinary at $p$, i.e., $a_{p}$ is not a $p$-adic unit, $L_{p}(f,\alpha)$ is no longer a measure and hence not an element of the Iwasawa algebra. 
Moreover, it has unbounded denominators and it is an element of a 
larger algebra known as the distribution algebra $\HH_{K}(\GG)$ (see section 2 for the definition).
When $a_{p}=0$, Pollack in \cite{Pollack} has given a remedy. 
If $\alpha_{1}, \alpha_{2}$ are the roots of $X^{2}+\epsilon(p)p^{k-1}$, Pollack showed that there exists a decomposition 
$$L_{p}(f,\alpha_{i}) = \log^{+}_{p,k}L_{p}^{+} + \alpha_{i}\log^{-}_{p,k}L_{p}^{-},$$
where $L_{p}^{\pm} \in \Lambda_{K}(\GG)$, for some finite extension $K$ of $\Qp$, and $\log^{\pm}_{p,k}$ are some power series in $\HH_{\Qp}(\GG_{1})$ depending only on $k$.
He also showed that if $k=2$, then $L_{p}^{\pm}$ have integral coefficients, i.e. they lie in $\Zp[\Delta][[\GG]]$.
Later Sprung (for $k=2$) 
in \cite{Spr1} and Lei--Loeffler--Zerbes (for $k\geq 2$) in \cite{LLZ1} have extended the work 
of Pollack when $a_{p}\neq 0$ using the method of logarithmic matrices. 

\begin{rem}
    On the algebraic side, we also have notions of signed Selmer groups due to Kobayashi (for $k=2, a_{p}=0$) in \cite{Koba}, Sprung (for $k=2, v_{p}(a_{p}) > 0$) in \cite{Spr1}, Lei (for $k\geq 2, a_{p}=0$) in \cite{Lei1}, and Lei--Loeffler--Zerbes (for $k\geq 2, v_{p}(a_{p})>0$) in \cite{LLZ1}.
    We also have signed Iwasawa main conjectures relating signed $p$-adic $L$-functions of non-ordinary modular forms with signed Selmer groups. See \cite{LLZ1} for more details.
\end{rem}

In this article, we extend the construction of signed $p$-adic $L$-functions (due to Pollack, Sprung, and Lei--Loeffler--Zerbes) to the setting of Bianchi modular forms using the two-variable $p$-adic $L$-functions constructed by Williams in \cite{Will}.
Bianchi modular forms are automorphic forms for $\mathrm{GL}(2)$ over quadratic imaginary fields.
Let $K$ be a quadratic 
imaginary field.
Fix a prime number $p\geq3$, which splits in $K$ as $(p)=\PP\PPP$, and let $k\geq 0$ be an integer.
Also, assume that $p$ does not divide the class number of $K$.
{
Let $\mathcal{G}$ be a cuspidal Bianchi eigenform over $K$ of weight $(k,k)$, level $\frkn$ such that $p$ divides $\frkn$ and $v_{p}(a_{\qq})<(k+1)$, where $a_{\qq}$ is the $U_{\qq}$ Hecke eigenvalue for all primes $\qq$ of $K$ which lie above $p$.
Then Williams has constructed a two-variable $p$-adic $L$-function $L_{p}(\mathcal{G})$ (see \cite[Theorem 7.4]{Will}) of a cuspidal Bianchi modular form $\mathcal{G}$ using overconvergent modular symbols.
More precisely, $L_{p}(\mathcal{G})$ is a locally analytic distribution over the 
ray class group $\text{Cl}(K,p^{\infty}) = G_{p^{\infty}}=\varprojlim_{n} G_{(p)^n}$, where $G_{(p)^n}$ is the ray class group of $K$ modulo $(p)^n$. 

We start with a Bianchi cuspform $\FF$ of weight $(k,k)$, level $\frkm$ coprime to $p$, and $\FF$ is good non-ordinary at both of the primes above $p$, that is, $v_{p}(a_{\PP})>0$ and $v_{p}(a_{\PPP})>0$, where $a_{\PP}$ and $a_{\PPP}$ are $T_{\PP}$ and $T_{\PPP}$ Hecke eigenvalues of $\FF$ respectively.
For $\qq\in\pset$, we assume $v_{p}(a_{\qqq}) > \left\lfloor \dfrac{k}{p-1} \right\rfloor$.
Moreover, let $\alpha_{\qq}$ and $\beta_{\qq}$ be the roots of Hecke polynomial $X^{2}-a_{\qq}X+p^{k+1}$ which we assume are distinct.
Then we get four $p$-stabilizations of $\FF$: $\FF^{\alpha_{\PP},\alpha_{\PPP}}, \FF^{\alpha_{\PP},\beta_{\PPP}}, \FF^{\beta_{\PP},\alpha_{\PPP}}$ and $\FF^{\beta_{\PP},\beta_{\PPP}}$, which are cuspidal Bianchi modular forms of 
the same weight as $\FF$ and level $p\mathfrak{m}$. 
Thanks to Williams, we can attach a two-variable $p$-adic $L$-function to each of the $p$-stabilizations $L_{*,\dagger} \coloneqq L_{p}(\FF^{*,\dagger})$, for $*\in\{\alpha_{\PP},\beta_{\PP}\}$ and $\dagger \in \{\alpha_{\PPP},\beta_{\PPP}\}$.
The main theorem of this article is:
\begin{thmx}[Theorem \ref{mainthm2}]
\label{thmA}
There exist 
two-variable power series with bounded coefficients, that is, there exist 
$L_{\sharp,\sharp}, L_{\sharp,\flat}, L_{\flat,\sharp}, L_{\flat,\flat}\in \Lambda_{E}(G_{p^{\infty}})$ such that
    \begin{equation*}
        \begin{pmatrix}
            L_{\alpha_{\PP},\alpha_{\PPP}} & L_{\beta_{\PP},\alpha_{\PPP}}\\[6pt]
            L_{\alpha_{\PP},\beta_{\PPP}} & L_{\beta_{\PP},\beta_{\PPP}}
        \end{pmatrix}
        = Q^{-1}_{\PPP}\underline{M_{\PPP}} \begin{pmatrix}
            L_{\sharp,\sharp} & L_{\flat,\sharp}\\[6pt]
            L_{\sharp,\flat} & L_{\flat,\flat}
        \end{pmatrix} (Q^{-1}_{\PP}\underline{M_{\PP}})^{T},
    \end{equation*}    
    where $\Mbarp$ and $\Mbarpp$ are  
    $2 \times 2$ logarithmic matrices, $Q_{\qqq} = \begin{pmatrix} \alpha_{\qqq} & -\beta_{\qqq} \\[6pt]
-p^{k-1} & p^{k-1} \end{pmatrix}$, 
and $\Lambda_{E}(G_{p^{\infty}}) \cong E[\Delta_{K}] \otimes_{\Oe} \Oe[[T_{1},T_{2}]]$, where $\Delta_{K}$ is a finite abelian group such that $G_{p^{\infty}} \cong \Delta_{K}\times \Zp^{2}$.
\end{thmx}

In the Appendix \ref{appA}, we have generalized Theorem \ref{thmA} from parallel weight cuspidal Bianchi modular forms to non-parallel $C$-cuspidal Bianchi modular forms. See Theorem \ref{mainthm3}.

\begin{rem}
In \cite{Loefller1}, Loeffler proved a special case of Theorem \ref{thmA}. 
He proved a Pollack style $\pm$-decomposition of two-variable unbounded $p$-adic $L$-functions attached to Bianchi modular forms of parallel weight $0$ with $a_{\qq}=0$, for $\qq\in\pset$. 
Even though the result was stated for Bianchi modular forms arising from the base change of weight $2$ elliptic modular forms to a quadratic imaginary field, the method works for all Bianchi modular forms as long as $a_{\qq}=0$ for $\qq\in\pset$.  
Pollack's distributions $\log^{\pm}_{\PP}, \log^{\pm}_{\PPP}$ over the ray class group $G_{p^\infty}$ for the decomposition were used in Loeffler's proof.
See \cite[Section 5]{Loefller1}.
In this article, we generalize \cite[Corollary 2]{Loefller1} to cuspidal Bianchi modular forms of parallel weight $k\geq 0$ and $v_{p}(a_{\qqq})>0$. We also generalize \cite[Proposition 9]{Loefller1}, \cite[Proposition 2.3, Proposition 2.5]{Lei3} from $k=0$ to $k\geq 0$.
We use a different approach from \cite{Loefller1}.
We construct and use Lei--Loeffler--Zerbes style $2\times 2$ logarithmic matrices (which generalize Pollack's $\pm$-logarithms)  $\Mbarp \text{ and } \Mbarpp$ to decompose $p$-adic $L$-functions with unbounded coefficients.
To construct these logarithmic matrices, we use $p$-adic Hodge theoretic tools.
For example, one of the key ingredients to construct $\Mbar$ is the Wach module basis due to Berger-Li-Zhu (in \cite{BLZ}).
The importance of Berger--Li--Zhu's construction is explained briefly in Section \ref{plan}.

\end{rem}
\subsection{Logarithmic matrices}
In this article, we construct logarithmic matrices in the sense of Sprung and Lei--Loeffler--Zerbes. 
Although we use $p$-adic Hodge theoretic tools to construct these matrices, one can think of them as purely algebraic elements. 
More generally, we construct $\Mbar\in M_{2,2}(\HH_{E}(\GG_1))$, where $\HH_{E}(\GG_1)$ is the distribution algebra over $E$, and $M_{2,2}(R)$ is the space of $2\times 2$ matrices with entries in $R$, having some growth properties. See Section \ref{sec5} for the details.

We do not have much information about $p$-adic Hodge theoretic properties of the Galois representations of Bianchi modular forms. 
If the level of the Bianchi modular form is away from $p$, we expect the corresponding Galois representation to be crystalline at $p$.
Only partial results are known due to Jorza, see \cite{Jorza1, Jorza2}.

We avoid or bypass the use of these conjectural properties of $p$-adic Galois representations associated with Bianchi modular forms by using Berger--Li--Zhu's construction. 
We explain it briefly here.
Since we have assumed $p$ splits as $(p)=\PP\PPP$ in $K$, the ray class group $G_{p^\infty}$ can be decomposed as $\Delta_{K}\times\GG_{\PP}\times\GG_{\PPP}$, where $\Delta_{K}$ is some finite abelian group and $\GG_{\PP}\cong\GG_{\PPP}\cong\Zp$. For $\qq\in\pset$, fix a topological generator $\gamma_{\qq}$ of $\GG_{\qq}$.
For a cuspidal Bianchi modular form $\FF$ of level $\frkm$ away from $p$, weight $(k,k)$ and trivial nebentypus, let $a_{\qq}$ be the $T_{\qq}$-eigenvalue of $\FF$. 
We assume $v_{p}(a_\qq)> \left\lfloor\dfrac{k}{p-2}\right\rfloor$. 
Moreover, let $\alpha_{\qq},\beta_{\qq}$ be the distinct roots of Hecke polynomial $X^{2}-a_{\PP}X+p^{k+1}$.
Then we construct a logarithmic matrix $\underline{M(\qq)}\in M_{2,2}(\GG_1)$, and then using \emph{change of variables}, i.e. by changing $\gamma_{0}$ with $\gamma_{\qq}$, we construct another matrix $\underline{M_{\qq}}\in M_{2,2}(\HH_{E}(\GG_\qq))$ with appropriate properties.
See Section \ref{modi} for the construction of $\underline{M_{\qq}}$ and Section \ref{twovar} for the details about the change of variable map.
Most importantly, the construction of $\underline{M_\qq}$ \emph{does not depend} on the $p$-adic Galois representation of $\FF$, but it depends only on the roots of Hecke polynomial $X^{2}-a_{\qq}X+p^{k+1}$ and the condition on $v_{p}(a_{\qq})$. Moreover, in Appendix \ref{appA}, we have constructed logarithmic matrices for $k$ and $\ell$, where $k$ need not be equal to $\ell$. 

\subsection{Plan of the article}
\label{plan}
We start with the setup and notations in Section 2 which we require throughout the article.

In Section 3, we recall tools from $p$-adic Hodge theory and Wach modules. 
In general, we look at  
crystalline representations, 
families of Wach modules and the relation between them due to Berger-Li-Zhu \cite{BLZ}.

In Section 4, we recall the exponential map constructed by Perrin-Riou and the $p$-adic regulator map along with Coleman maps, which were introduced by Lei--Loeffler--Zerbes in \cite{LLZ1} and \cite{LLZ2}. 
Moreover, we study the relationship between the exponential map and the $p$-adic regulator map. In this section, we also introduce the logarithm matrix $\Mbar$, which is an element of $M_{2,2}(\HH_{E}(\GG_1))$. 

Section 5 deals with the factorization of power series in one variable.
We first investigate $\Mbar$ in more depth.
Then we prove the following result using $\Mbar$: 
\begin{thmx}[Theorem \ref{thm1}]
\label{thmB}
 Let $E$ be a finite extension of $\Qp$. Let $\alpha, \beta \in \Oe$ and $k\geq 2$ be an integer such that $\alpha\beta= p^{k-1}$. 
 Assume 
 $\alpha \neq \beta$ and $v_{p}(\alpha + \beta) >  \left\lfloor \dfrac{k-2}{p-1} \right\rfloor$. 
For each $\lambda\in\{\alpha,\beta\}$, let $F_{\lambda} \in \HH_{E,v_{p}(\lambda)}(\GG)$, 
  such that for any integer $0\leq j \leq k-2$, and for any Dirichlet character $\omega$ of conductor $p^n$, we have $F_{\lambda}(\chi^{j}\omega)=\lambda^{-n}C_{\omega,j}$, where $\chi$ is the $p$-adic cyclotomic character 
 and $C_{\omega,j}\in \overline{\Qp}$ that is independent of $\lambda$. Then there exist 
 $F_{\flat}, F_{\sharp} \in \Lambda_{E}(\GG)$ such that
    \begin{equation*}
        \begin{pmatrix} F_{\alpha} \\[6pt] F_{\beta} \end{pmatrix} = Q^{-1}\Mbar \begin{pmatrix} F_{\sharp} \\[6pt] F_{\flat} \end{pmatrix}.
    \end{equation*}
    Note that $\Mbar$ depends on $\alpha+\beta$ and $k$, and the matrix $Q$ depends on $\alpha$ and $\beta$.
\end{thmx}
\begin{rem}
 In \cite[Section 2]{BL1}, the authors proved a similar result as above under the Fontaine-Laffaille 
 condition ($p>k$). 
 In this article, we are replacing this condition with a weaker condition $v_{p}(\alpha + \beta) > \left\lfloor \dfrac{k-2}{p-1} \right\rfloor$.
We also use different methods than the methods used in \cite{BL1}. For example, we obtain properties of $\Mbar$ using the $p$-adic regulator $\LT$ and Perrin-Riou's exponential map $\Omega_{T_W}$.  
\end{rem}

\begin{rem}
    Theorem \ref{thmB} can be used in the decomposition of any two power series satisfying certain growth conditions and interpolation properties. 
    In \cite{Deo1}, which is an upcoming work regarding $p$-adic Asai $L$-distributions of $p$-non-ordinary Bianchi modular forms, we can use this method to decompose the distributions into bounded measures.
\end{rem}

In Section 6, we develop the two-variable setup and recall definitions of ray class groups, Hecke characters, etc.

Bianchi modular forms and their $p$-adic $L$-functions are  
briefly recalled in Section 7.

In the last section, we prove the main theorem (Theorem \ref{thmA}) of this article.
We generalize and apply 
results of \cite{Lei3} in 
our setting.

\subsection*{Acknowledgments}
I would like to thank my PhD advisor, Antonio Lei, for suggesting this problem to me, for answering all my questions, and for his patience, guidance, encouragement, and support.
I would like to thank Luis Santiago Palacios for the useful discussions about \cite{LPS2} and clarifications regarding $C$-cuspidal Bianchi modular forms.
I would like to thank my brother Shaunak Deo for helpful mathematical and non-mathematical conversations. 
I would also like to thank Katharina M\"uller for her suggestions on the earlier draft of this article. 
Finally, I would like to thank anonymous referees for their valuable comments and suggestions, which greatly helped in improving the exposition.

\section{Setup and notations}
\label{setup}
Fix an odd prime $p$. 
Let $E$ be a finite extension of $\Qp$ with the ring of integers $\Oe$. 
Let $\alpha, \beta \in \Oe$ 
such that $v_{p}(\alpha+\beta)>0$, 
and there exists $v\in\Oe^\times$ and an integer $k\geq 2$ such that $\alpha\beta= v p^{k-1}$.
We assume that $v^{1/2} \in \Oe^{\times}$.
Denote $a=\alpha+\beta.$
We denote $\Gal(\overline{\Qp}/\Qp)$ by $G_{\Qp}$. 
\begin{assum}
     $v_{p}(a) > m = \left\lfloor \dfrac{k-2}{p-1} \right\rfloor$ and $\alpha \neq \beta$. 
\end{assum}    
We fix $a, \alpha,\beta,$ and $v$ for the rest of the article.

\subsubsection*{Iwasawa algebras} Let $\QQ_{p,n} = \QQ_{p}(\mu_{p^n})$, where $\mu_{p^n}$ is the set of all $p^{n}$-th roots of unity. 
Let $\QQ_{p,\infty} = \bigcup_{n \geq 1} \QQ_{p,n}$.
Then $\GG = \Gal(\QQ_{p,\infty}/\Qp) \cong \Delta \times \ZZ_p$, 
where $\Delta$ is the torsion group of $\GG$ of order $p-1$. Let $\GG_1$ be a subgroup of $\GG$ such that $\GG_1 \cong \GG/\Delta \cong \ZZ_p$.
In other words, $\GG_{1}$ is the Galois group of $\QQ_{p,\infty}$ over $\QQ_{p,1}$.
We denote the Iwasawa algebra $\Oe \otimes_{\Zp} \ZZ_{p}[[\GG]] \cong \Oe[[\GG]]$ over $\Oe$ by $\Lambda_{\Oe}(\GG)$. 
Fix a topological generator $\gamma_0$ of $\GG_1$. Then we can identify $\Oe[[\GG_1]]$ with $\Oe[[X]]$ via identification $\gamma_{0} \mapsto 1+X$. 
This can be extended to $\Lambda_{\Oe}(\GG) \cong \Oe[\Delta][[X]] $.
We further write $\Lambda_{E}(\GG_1) = E\otimes_{\OO_E} \Lambda_{\Oe}(\GG_1)$ and $\Lambda_{E}(\GG) = E\otimes_{\OO_E} \Lambda_{\Oe}(\GG)$.  
Fix a topological generator $u$ of $1+p\ZZ_p$ and let $\chi$ be the $p$-adic 
cyclotomic character on $\GG$ such that $\chi(\gamma_{0}) = u$. 

\subsubsection*{Power series rings} Given any power series $F \in E[[X]]$ and $0 < \rho <1$, we define the sup norm $\lv F \lv_{\rho} = \text{sup}_{\vert z \vert_p \leq \rho} \vert F(z) \vert_{p}$. 
For any real number $r \geq 0$, we define
$$\HH_{r} = \{ F \in E[[X]] : \supt (p^{-tr}\lv F \lv_{\rho_t}) < \infty\},$$
where $\rho_{t}= p^{-1/p^{t-1}(p-1)}$ and $t \geq 1$ is an integer. 
Equivalently, we have
\[\HH_{r}=\left\{ F(X)=\sum_{n\geq 0}c_{n}X^n \in E[[X]] : \sup_{n}\dfrac{|c_{n}|_{p}}{n^r} < \infty\right\}.\]
If $F(X)\in \HH_{r}$, then $F$ is $O(\log_{p}^{r})$, that is 
\[\lv F\lv_{\rho} = O\left(\lv \log_{p}^{r}(1+X))\lv_{\rho}\right)\]
as $\rho\to 1^{-}$.
We write $\HH_{E} = \bigcup_{r\geq 0}\HH_r$. 

We define $\HH_{E,r}(\GG)$ to 
be the set of power series 
$\sum_{n \geq 0,\sigma \in \Delta} c_{n,\sigma}\cdot\sigma\cdot (\gamma_{0}-1)^{n},$
such that $\sum_{n\geq 0}c_{n,\sigma}X^{n} \in \HH_r$ for all $\sigma \in \Delta$. 
In other words, the elements of $\HH_{E,r}(\GG)$ are the power series in $\gamma_{0}-1$ over $E[\Delta]$ with the growth rate $O(\log_{p}^r)$. 
Write $\HH_{E}(\GG) = \bigcup_{r\geq 0}\HH_{E,r}(\GG)$. 
We call $\HH_{E}(\GG)$ \emph{the space of distributions on} $\GG$. 
We can identify $\HH_{E}(\GG)$ with
$$\left\{ F(X)\green{=}\sum_{\substack{n\geq0,\\ \sigma\in\Delta}} c_{n,\sigma}\cdot\sigma\cdot X^n \in E[\Delta][[X]] : \sum_{n\geq 0}c_{n,\sigma}X^n\text{ converges for all }X\in\CC_p \text{ with } |X|_{p} < 1\right\}$$
where $X$ corresponds to $\gamma_{0}-1$. 
\begin{rem}
    In \cite{Kato} and \cite{PR1}, the spaces $\HH_r$ are defined using $\lim_{n} n^{-r}|c_{n}|_p = 0$ instead of $\sup_{n} (n^{-r}|c_{n}|_{p})<\infty$. 
    For example, if $r=0$, then by the notation in \cite{Kato}, $\HH_{0}=E\langle X\rangle$, where $E\langle X \rangle$ is the one-variable Tate algebra. 
    But, in our context, we have the following identification
    \[\HH_{0}= E\otimes_{\OO_E}\OO_{E}[[T]].\]
    We do not need the stronger notion of Tate algebras in this article.
    See also \cite[Lemmas 3.2, 5.2]{Pollack}.
\end{rem}

\subsubsection*{Fontaine's rings} Let $\pi$ be a variable, $\mathbb{A}_{\Qp}^{+} = \ZZ_{p}[[\pi]]$ and $\mathbb{B}^{+}_{\Qp}= \mathbb{A}_{\Qp}^{+}[1/p]$.
Let $\mathbb{A}_{\Qp}$  be the ring of Laurent series $\sum_{i=-\infty}^{+\infty} a_{i}\pi^{i}$ such that $a_{i}\in \Zp$ and $a_{i} \to 0$ as $i\to -\infty$. 
Write $\BR^{+}_{\mathrm{rig},\Qp}$ for the ring of power series $f(\pi) \in \Qp[[\pi]]$ such that $f(X)$ converges everywhere in the open unit $p$-adic disk. 
We equip $\BR^{+}_{\mathrm{rig},\Qp}$ with actions of a Frobenius operator $\vp$ and $\GG$ by $\vp: \pi \mapsto (1+\pi)^{p}-1$ and $\sigma: \pi \mapsto (1+\pi)^{\chi(\sigma)} -1$ for all $\sigma \in \GG$. 
We then write $\BrigE$ for the power series ring $E\otimes\Brig$.
We can define a left inverse $\psi$ of $\vp$ such that 
$$\vp\circ\psi(f(\pi)) = \dfrac{1}{p} \sum_{\zeta^{p} = 1}f(\zeta(1+\pi) - 1).$$
Inside $\BrigE$, we have subrings $\AAA^{+}_{E} = \Oe[[\pi]]$ and $\BR^{+}_{E} = E \otimes \AAA^{+}_{E}$. 
The actions of $\vp, \psi$, and $\GG$ preserve 
these subrings.
Write $t=\log(1+\pi) \in \BrigE$ and $q=\vp(\pi)/\pi \in \AAA^{+}_{E}$.
Note that $\vp(t) = pt$ and $\sigma(t) = \chi(\sigma)t$ for all $\sigma \in \GG$, since $\log(1+\pi)=\pi\displaystyle\prod_{n\geq 1}\dfrac{\vp^{n-1}(q)}{p}$.

\subsubsection*{Mellin transform} We have a $\Lambda_{E}(\GG)$-
module isomorphism between $\HH_{E}(\GG)$ and $(\BrigE)^{\psi=0}$ due to the action of $\GG$ on $\BrigE$, called the Mellin transform.
The isomorphism is given by 
\begin{align*}
    \mathfrak{M}: \HH_{E}(\GG) &\to (\BrigE)^{\psi=0}\\
                   f(\gamma_{0}-1) &\mapsto f(\gamma_{0}-1)\cdot(1+\pi).
\end{align*}
Moreover, $\Lambda_{\Oe}(\GG)$ corresponds to $(\AAA^{+}_{E})^{\psi=0}$ and $\intiwG$ corresponds to $(1+\pi)\vp(\AAA^{+}_{E})$ under $\mathfrak{M}$. 
Let $\HH_{E}(\GG_1) = \{ f(\gamma_{0}-1) : f \in \HH_{E}\}$, then $\HH_{E}(\GG_1)$ corresponds to $(1+\pi)\vp(\BrigE)$. 
See \cite[Section B.2.8]{PR2} for more details.

\section{Crystalline representations and Wach modules}
\label{sec3}
In this section, we recall definitions of crystalline representations and Wach modules. Furthermore, we recall the construction of families of Wach modules from \cite{BLZ}. The primary reference for this section is \cite[Sections 1, 2, and 3]{BLZ}.

\subsection{Crystalline representations}
Let $\mathbb{B}_{\mathrm{crys}}$ be Fontaine's crystalline period ring. 
Recall that we call a $\Qp$-linear $G_{\Qp}$-representation $V$ a crystalline representation if $V$ is $\mathbb{B}_{\mathrm{crys}}$-admissible.
In other words, $V$ is a crystalline representation if the dimension of the filtered $\vp$-module $\Dcris(V)=(\mathbb{B}_{\mathrm{crys}}\otimes V)^{G_{\Qp}}$ is $\text{dim}_{\Qp}V$.
For any integer $j$, we take $\Qp(j)=\Qp\cdot e_{j}$, where $G_{\Qp}$ acts on $e_{j}$ via $\chi^j$. We know that $\Qp(j)$ is a crystalline representation.
Then for any crystalline representation $V$, the representation $V(j)=V(\chi^{j})=V\otimes \Qp(j)$ is again a crystalline representation.
Moreover,we have $\Dcris(V(j)) = t^{-j}\Dcris(V)\otimes e_{j}$.   
We say a crystalline (or more generally a Hodge--Tate) representation $V$ is positive if its Hodge--Tate weights are negative.
Note that we are assuming the $p$-adic cyclotomic character has Hodge--Tate weight $+1$.

Let $E$ be a finite extension of $\Qp$.
We say that an $E$-linear $G_{\Qp}$-representation $V$ is crystalline if and only if the underlying $\Qp$-linear representation is crystalline.
In this case, $\Dcris(V)$ is an $E$-vector space with $E$-linear Frobenius and a filtration of $E$-vector spaces.
More precisely, $\Dcris(V)$ is an admissible $E$-linear filtered $\vp$-module and the functor $V \mapsto \Dcris(V)$ is an equivalence of categories from the category of crystalline $E$-linear representations to the category of admissible $E$-linear filtered $\vp$-modules (see \cite{CF} for more details).

\subsubsection{Crystalline representations as filtered $\vp$-modules}
Let $D_{k,v^{1/2}a}$ be a filtered $\vp$-module given by $D_{k,v^{1/2}a} = Ee_{1} \oplus Ee_{2}$ where:
\begin{align*}
\begin{matrix}
    \begin{cases} \vp(e_{1})= p^{k-1}e_{2} \\ \vp(e_{2})=-e_{1} + (v^{1/2}a) e_2 \end{cases} & \text{and } & 
    \mathrm{Fil}^{i}D_{k,v^{1/2}a}=\begin{cases} D_{k,v^{1/2}a} \text{ if } i \leq 0 \\ Ee_1 \text{ if }1\leq i \leq k-1 \\ 0 \text{ if } i \geq k.\end{cases}
\end{matrix}
\end{align*}
Take $e'_{1}= v^{1/2}e_{1}$ and $e'_{2}=e_2$. 
Thus, $e'_{1}, e'_{2}$ is another $E$-basis of $D_{k,v^{1/2}a}$.
The matrix of $\vp$ with respect to basis $e'_{1}, e'_{2}$ is
$$\tilde{A}_{\vp} = \begin{pmatrix} 0 & -v^{-1/2} \\[6pt] v^{1/2}p^{k-1} & v^{-1/2}a \end{pmatrix}.$$ 

\begin{thm}[Colmez-Fontaine{{\cite{CF}}}, Berger-Li-Zhu{{\cite{BLZ}}}]
\label{thm1.1}
 There exists a crystalline $E$-linear representation $V_{k,v^{1/2}a}$, such that $\Dcris(V^{*}_{k,v^{1/2}a}) = D_{k,v^{1/2}a}$, where $V^{*}_{k,v^{1/2}a} = \mathrm{Hom}(V_{k,v^{1/2}a}, E)$.
\end{thm}
\begin{proof}
See \cite[Section I and Proposition 3.2.4]{BLZ}
\end{proof}
From the above theorem, we get
\[\Dcris(V^{*}_{k,v^{1/2}a}) = Ee_{1} \oplus Ee_{2} = Ee'_{1} \oplus Ee'_{2}.\]
The Hodge--Tate weights of $V_{k,v^{1/2}a}$ are $0$ and $k-1$, and thus the Hodge--Tate weights of $V^{*}_{k,v^{1/2}a}$ are $0$ and $1-k$. 
Let $W=V^{*}_{k,v^{1/2}a}(\chi^{k-1}\otimes\eta)$, where 
$\eta:G_{\Qp} \to \overline{\Qp}^{\times}$ is an unramified character such that $\eta(\text{Frob}_p)=v^{1/2}$. 
Therefore $W$ is a crystalline representation with Hodge--Tate weights $0$ and $k-1$. 

Let $w_{i} = e'_{i}\otimes t^{-(k-1)}e_{k-1}\otimes e_{\eta}$, for $i =1,2$, where $e_{\eta}$ is a basis of $\Qp(\eta)$ and the action of $\vp$ on $e_\eta$ is given by $\vp(e_{\eta}) = \eta(\text{Frob}^{-1}_p)e_{\eta}$. 
Then $w_{1}, w_{2}$ is a basis of $\Dcris(W)$.
The action of $\vp$ on $w_i$ can be calculated as
\begin{align*}
\begin{cases}
\vp(w_{1})= w_2 \\ \vp(w_{2})=(-1/v p^{k-1})w_{1} + (a/v p^{k-1}) w_2 .
\end{cases}
\end{align*}
Thus, the matrix of $\vp$ with respect to basis $w_{1}, w_{2}$ is $$A_{\vp}=\begin{pmatrix} 0 & \dfrac{-1}{vp^{k-1}} \\[12pt]  1 & \dfrac{a}{vp^{k-1}}\end{pmatrix}.$$

\subsection{Wach modules}
An \'etale $(\vp, \GG)$-module over $\mathbb{A}_{\Qp}$ is a free $\mathbb{A}_{\Qp}$- module $M$ of finite rank, with semilinear action of $\vp$ and a continuous action of $\GG$ commuting with each other, such that $\vp(M)$ generates $M$ as an $\mathbb{A}_{\Qp}$-module.
In \cite[A.3.4]{Fon1}, Fontaine has constructed a functor $T \mapsto \mathbb{D}(T)$ which associates to every $\Zp$-linear representation an \'etale $(\vp,\GG)$-module over $\mathbb{A}_{\Qp}$.
The $(\vp,\GG)$-module $\mathbb{D}(T)$ is defined as $(\mathbb{A}\otimes T)^{H_{\Qp}}$, where $\mathbb{A}$ is the ring defined in \cite{Fon1} and $H_{\Qp}=\Gal(\overline{\Qp}/\QQ_{p,\infty})$.
He also shows that the functor $T \mapsto \mathbb{D}(T)$ is an equivalence of categories.
By inverting $p$, we also get an equivalence of categories between the category of $\Qp$-linear $G_{\Qp}$-representations and the category of \'etale $(\vp,\GG)$-modules over $\mathbb{B}_{\Qp}=\mathbb{A}_{\Qp}[p^{-1}]$. 

If $E$ is a finite extension of $\Qp$, we extend the Frobenius and the action of $\GG$ to $E\otimes\mathbb{B}_{\Qp}$ by $E$-linearity.
We then get an equivalence of categories between $\Oe$-modules (or $E$-linear $G_{\Qp}$-representations) and the category of $(\vp,\GG)$-modules over $\Oe\otimes\mathbb{A}_{\Qp}$ (or over $E\otimes\mathbb{B}_{\Qp}$), given by $T\mapsto \mathbb{D}(T)$. 

In \cite{BER1}, Berger shows that if $V$ is an $E$-linear $G_{\Qp}$-representation, then $V$ is crystalline with Hodge--Tate weights in $[a,b]$ if and only if there exists a unique $E\otimes \mathbb{B}^{+}_{\Qp}$-module $\NN(V) \subset \mathbb{D}(V)$ such that:
\begin{enumerate}
    \item $\NN(V)$ is free of rank $d=\dim_{E}(V)$ over $E\otimes\mathbb{B}^{+}_{\Qp}$;
    \item The action of $\GG$ preserves $\NN(V)$ and is trivial on $\NN(V)/\pi\NN(V)$;
    \item $\vp(\pi^{b}\NN(V)) \subset \pi^{b}\NN(V)$ and $\pi^{b}\NN(V)/\vp^{*}(\pi^{b}\NN(V))$ is killed by $q^{b-a},$ where $q=\frac{\vp(\pi)}{\pi}$ and $\vp^{*}(\pi^{b}(\NN(V)))$ is $\mathbb{B}^{+}_{E}$-submodule of $\NN(V)[\pi^{-1}]$ generated by $\vp(\pi^{b}\NN(V))$.
\end{enumerate}
Moreover, if $V$ is crystalline and positive, then we can take $b=0$. 
In this case, $\NN(V)/\pi\NN(V)$ is a filtered $E$-module and there exists an isomorphism $\NN(V)/\pi\NN(V) \cong \Dcris(V)$ as filtered $\vp$-modules.
See \cite[Section III.4]{BER1} 
for more details. 

Let $T$ be a $G_{\Qp}$-stable lattice in $V$. Then $\NN(T)\coloneqq\mathbb{D}(T) \cap \NN(V)$ is an $\OO_{E}\otimes \mathbb{A}^{+}_{\Qp}$-lattice in $\NN(V)$. 
By \cite{BER1}, the functor $T \mapsto \NN(T)$ gives a bijection between the $G_{\Qp}$-stable lattices $T$ in $V$ and the $\OO_{E}\otimes \mathbb{A}^{+}_{\Qp}$-lattices in $\NN(V)$ which satisfy
\begin{enumerate}
     \item $\NN(T)$ is free of rank $d=\dim_{E}(V)$ over $\OO_{E}\otimes\mathbb{A}^{+}_{\Qp}$;
    \item The action of $\GG$ preserves $\NN(T)$;
    \item $\vp(\pi^{b}\NN(T)) \subset \pi^{b}\NN(T)$ and $\pi^{b}\NN(T)/\vp^{*}(\pi^{b}\NN(T))$ is killed by $q^{b-a}$, where $\vp^{*}(\pi^{b}(\NN(T)))$ is $\mathbb{A}^{+}_{E}$-submodule of $\NN(T)[\pi^{-1}]$ generated by $\vp(\pi^{b}\NN(T))$. 
\end{enumerate}
The $E\otimes\mathbb{B}^{+}_{\Qp}$-module $\NN(V)\subset \mathbb{D}(V)$ as well as $\mathcal{O}_{E}\otimes \mathbb{A}^{+}_{\Qp}$-module $\NN(T)\subset \mathbb{D}(T)$ are called \emph{Wach modules}.

\subsubsection{Families of Wach modules}
\label{secfam}
In this section, we recall some results from \cite{BLZ}.

Recall $q=\vp(\pi)/\pi$. 
We define $\lambda_{+}$ and $\lambda_{-}$ as
$$\lambda_{+} = \prod_{n\geq 0} \dfrac{\vp^{2n+1}(q)}{p}$$
and
$$\lambda_{-} = \prod_{n\geq 0} \dfrac{\vp^{2n}(q)}{p}.$$

\begin{lem}
\label{lem1.2}
Write $p^{m}(\lambda_{-}/\lambda_{+})^{k-1} = \sum_{i\geq 0} z_{i}\pi^{i}$, where $m= \left\lfloor \dfrac{k-2}{p-1} \right\rfloor$ and define $z = z_{0} + z_{1}\pi + \cdots + z_{k-2}\pi^{k-2}$. 
Then $z\in \ZZ_{p}[[\pi]]$. 
\end{lem}
\begin{proof}
See \cite[Proposition 3.1.1]{BLZ}. 
\end{proof}

Let $Y$ be a variable.
Define a matrix
$$P(Y) = \begin{pmatrix} 0  & -v^{-1/2} \\[6pt] v^{1/2}q^{k-1} & Yv^{-1/2}z\end{pmatrix}.$$
Then by \cite[Proposition 3.1.3]{BLZ}, for $\gamma \in \GG$, there exists a matrix $G_{\gamma}(Y) \in I_{2} + \pi M_{2,2}(\ZZ_p[[\pi, Y]])$ such that 
$$P(Y)\vp(G_{\gamma}(Y)) = G_{\gamma}(Y)\gamma(P(Y)).$$
Note that $\vp$ and $\gamma\in\GG$ acts trivially on $Y$. 

\begin{lem}
\label{lem1.3}
For $\delta= \dfrac{a}{p^m}$ and $\gamma, \gamma' \in \GG$, we have $G_{\gamma\gamma'}(\delta) = G_{\gamma}(\delta)\gamma'(G_{\gamma'}(\delta))$ and $P(\delta)\vp(G_{\gamma}(\delta))=G_{\gamma}(\delta)\gamma(P(\delta))$. 
Therefore, one can use the matrices $P(\delta)$ and $G_{\gamma}(\delta)$ to define a Wach module $\mathbb{N}_{k}(\delta)$ over $\Oe[[\pi]]$.
\end{lem}
\begin{proof}
Define the free $\Oe[[\pi]]$-module of rank $2$ with basis $n_1, n_2$ as: $\NN_{k}(\delta) = \Oe[[\pi]]n_1 \oplus \Oe[[\pi]]n_2$. 
Endow it with Frobenius $\vp$ and an action of $\gamma \in \Gamma$ such that the matrix of $\vp$ with respect to the basis $n_{1}, n_{2}$ is $P(\delta) = \begin{pmatrix} 0  & -v^{-1/2} \\[6pt] v^{1/2}q^{k-1} & v^{-1/2}\delta\cdot z\end{pmatrix}$ and the matrix of $\gamma$ is $G_{\gamma}(\delta)$. 
See \cite[Proposition 3.2.1]{BLZ} 
for details. 
\end{proof}

The above lemma implies $E\otimes\NN_{k}(\delta)=\NN(V^{*}_{k,v^{1/2}a})$, where $\delta=a/p^m$, and $V_{k,v^{1/2}a}$ is the crystalline $E$-linear representation in Theorem \ref{thm1.1}.
Here $\NN(V^{*}_{k,v^{1/2}a})$ is the Wach module associated to the crystalline representation $V^{*}_{k,v^{1/2}a}$. More precisely:

\begin{thm}
\label{thm1.4}
    The filtered $\vp$-module $E\otimes_{\Oe}(\NN_{k}(\delta)/\pi\NN_{k}(\delta))$ is isomorphic to the $\vp$-module $D_{k,v^{1/2}a}$ defined in the Theorem \ref{thm1.1}. 
\end{thm}
\begin{proof}
This can be proved using \cite[Proposition 3.2.4]{BLZ}. 
\end{proof}

We adapt the above machinery in our setting.
Recall that $W=V^{*}_{k,v^{1/2}a}(\chi^{k-1}\otimes \eta)$. 
Denote by $T_W$ the $\Oe$-lattice $T(\chi^{k-1}\otimes \eta)$ in $W$, 
where $T$ is an $\Oe$-lattice in $V_{k,v^{1/2}a}$ such that $\NN_{k}(\delta)= \NN(T^{*})$. 
By an abuse of notation, we write $\NN_{k}(\delta) = \NN(T_W) = \Oe[[\pi]]n'_{1} \oplus \Oe[[\pi]]n'_{2}$, where $n'_{1}, n'_{2}$ is a basis after twisting the basis $n_{1},n_{2}$ of $\NN(T)$ with $\chi^{k-1}\otimes \eta$.  
Then the matrix of $\vp$ with respect to $\{n'_{1}, n'_{2}\}$ is 
$$P= \begin{pmatrix} 0 & -1/vq^{k-1} \\[6pt] 1 & \delta\cdot z/vq^{k-1} \end{pmatrix}.$$
Note that $P \equiv A_{\vp} \modd \pi,$
since $q \equiv p \modd \pi$ and $\delta\cdot z \equiv a \modd \pi$.
We fix this $\Oe[[\pi]]$-basis $n'_{1},n'_{2}$ for $\NN(T_{W})$ for the rest of the article.


\section{Perrin-Riou's Exponential map, Coleman maps, and the $p$-adic regulator}
\label{sec4}
We recall definitions of Perrin-Rious's exponential map, the $p$-adic regulator, and Coleman maps.
We also explicitly study these maps and the relationship between them after fixing some basis. The primary reference for this section is \cite{LLZ2}.

\subsection{Iwasawa cohomology and Wach modules}
Let $V$ be any crystalline $E$-linear representation of $G_{\Qp}$ and let $T$ be an $\Oe$-lattice inside $V$. 
The Iwasawa cohomology group $H^{1}_{\mathrm{Iw}}(\Qp, T)$ is defined by
$$H^{1}_{\mathrm{Iw}}(\Qp,T) = \varprojlim H^{1}(\QQ_{p,n},T),$$
where the inverse limit is taken with respect to the corestriction maps.
Then, due to Fontaine (see \cite[Section II.1]{CC}), there exists a canonical $\Lambda_{\Oe}(\GG)$-module isomorphism
\begin{equation*}
    h^{1}_{\mathrm{Iw}}\colon \mathbb{D}(T)^{\psi=1} \to H^{1}_{\mathrm{Iw}}(\Qp,T), \label{1}
\end{equation*}
where $\mathbb{D}(T)$ is the $(\vp,\Gamma)$-module associated to $T$. 

From now on, we consider the $p$-adic representation $W=V^{*}_{k, v^{1/2}a}(\chi^{k-1}\otimes\eta)$ and the $\OO_E$-lattice $T_W$ in $W$ studied in Section \ref{secfam} unless mentioned otherwise. 
Moreover, let $\Dcris(T_W)$ be the image of $\NN(T_W)/\pi\NN(T_W)$ in $\Dcris(W).$ 
Then  \begin{enumerate}
    \item $\Dcris(T_W)$ is filtered $\vp$-module over $\Oe$, 
    \item $\Dcris(T_W)= \Oe\cdot w_1 \oplus \Oe\cdot w_2$,
    \item the matrix of $\vp$ with respect to the basis $w_1, w_2$ is $A_{\vp}$. 
\end{enumerate}

For the representation $W$, the eigenvalues of the $\vp$ are $\alpha^{-1}, \beta^{-1}$. 
From now on we assume that $\alpha^{-1}$ and $\beta^{-1}$ are not integral powers of the prime $p$. 
Since the Hodge--Tate weights of $W$ are $0$ and $k-1$, we have the following theorem due to Berger: 

\begin{thm}[Berger, {{\cite[Theorem A.3]{BER2}}}]
\label{thm1.5}
    For the $G_{\Qp}$-stable $\Oe$-lattice $T_W$ in $W$, there exists a $\Lambda_{\Oe}(\GG)$-module isomorphism 
     \[
     h^{1}_{\mathrm{Iw},T_W}\colon \NN(T_W)^{\psi=1} \to H^{1}_{\mathrm{Iw}}(\Qp,T_{W}).
     \]   
    Moreover, we can extend this isomorphism from $\Lambda_{\Oe}(\GG)$-modules to $\Lambda_{E}(\GG)$-modules
     \[
      h^{1}_{\mathrm{Iw},W}\colon \NN(W)^{\psi=1} \to \Hiw,
     \]    
    where $\NN(W)=E\otimes \NN(T_W)$ and $\Hiw = E \otimes H^{1}_{\mathrm{Iw}}(\Qp,T)$.
\end{thm}
\begin{rem}For a given $p$-adic $E$-linear $G_{\Qp}$-representation $V$, the torsion part of $\mathrm{H}^{1}_{\mathrm{Iw}}(\Qp, V)$ is given by $V^{\Gal(\overline{\Qp}/\Qp(\mu_{p^\infty}))}$.
But, for our representation $W$, by \cite[Theorem 4.1.1]{BLZ}, $W^{\Gal(\overline{\Qp}/\Qp(\mu_{p^\infty}))}=0$, and hence the Iwasawa cohomology $\Hiw$ is a free $\Lambda_{E}(\GG)$-module of rank $2$.\end{rem}

\subsection{Coleman maps}
For the $p$-adic representation $W$ and the $\Oe$-lattice $T_W$ in $W$, we deduce $\NN(T_W) \subset \vp^{*}(\NN(T_W))$, since the Hodge--Tate weights of $W$ are non-negative.
Here $\vp^{*}(\NN(T_W))$ is $\mathbb{A}^{+}_{E}$-submodule of $\NN(T_W)[\pi^{-1}]$ generated by $\vp(\NN(T_W))$ (See \cite[Lemma 1.7]{LLZ2}). 
Hence there exists a well-defined map $1-\vp\colon \NN(T_W) \to \vp^{*}(\NN(T_W))$ which maps $\NN(T_W)^{\psi=1}$ to $(\vp^{*}\NN(T_W))^{\psi=0}$. 

\begin{thm}[Lei--Loeffler--Zerbes, Berger]
\label{thm1.6}
    The $\Lambda_{\Oe}(\GG)$-module $(\vp^{*}\NN(T_W))^{\psi=0}$ is free of rank $2$. Moreover, for any basis $v_1, v_2$ of $\Dcris(T_W)$, there exists an $\Oe\otimes \mathbb{A}^{+}_{\Qp}$-basis $n_{1},n_{2}$ of $\NN(T_W)$ such that $n_{i} \equiv v_{i} \modd \pi $ and $(1+\pi)\vp(n_{1}),(1+\pi)\vp(n_{2})$ form a $\Lambda_{\Oe}(\GG)$-basis of $(\vp^{*}\NN(T_W))^{\psi=0}$.
\end{thm}

\begin{proof}
    See\cite[Lemma 3.15]{LLZ1} 
    for the proof for any crystalline representation of dimension $d$.
\end{proof}
The above theorem gives an isomorphism of $\Lambda_{\Oe}(\GG)$-modules
 \[
   \mathfrak{J}\colon (\vp^{*}\NN(T_W))^{\psi=0} \to \Lambda_{\Oe}(\GG)^{\oplus 2}. 
 \]   

\begin{defi}[The Coleman map]
We define the Coleman map
\begin{equation*}
    \Col=(\Col_{i})_{i=1}^{2}\colon \NN(T_W)^{\psi=1} \to \Lambda_{\Oe}(\GG)^{\oplus 2}
\end{equation*}
as the composition $\mathfrak{J}\circ (1-\vp)$. 
\end{defi}
This Coleman map $\Col$ can be extended to a map from $\NN(W)$ to get a $\Lambda_{E}(\GG)$-module homomorphism
\begin{equation*}
    \Col\colon \NN(W)^{\psi=1} \to \Lambda_{E}(\GG)^{\oplus 2}.
\end{equation*}

From the above discussion, for the fixed basis $n'_1, n'_2$ for $\NN(T_W)$ and basis $w_{1},w_{2}$ for $\Dcris(T_W)$, we get a matrix $\Mbar$ as follows: The elements $(1+\pi)(\vp(n'_{1})),(1+\pi)(\vp(n'_{2}))$ form a $\Lambda_{\Oe}(\GG)$-basis of $(\vp^{*}\NN(T_{W}))^{\psi=0}$. Furthermore the elements $(1+\pi)\otimes w_{1},(1+\pi)\otimes w_{2}$ form a basis of $(\BrigE)^{\psi=0} \otimes \Dcris(T_W)$ as a $\HH_{E}(\GG)$-module. Since $n'_{i} \equiv w_{i} \mod \pi$ for $i=1,2$, there exists a unique $2\times2$ matrix $\Mbar \in M_{2,2}(\HH_{E}(\GG))$ such that
\begin{equation}
    \begin{bmatrix}(1+\pi)\vp(n'_{1}) & (1+\pi)\vp(n'_{2}) \end{bmatrix} = \begin{bmatrix}(1+\pi)\otimes w_{1} & (1+\pi)\otimes w_{2} \end{bmatrix}\Mbar. 
\end{equation}
That is, $\Mbar$ is a change of basis matrix for the following homomorphism of $\HH_{E}(\GG)$-modules:
\[(\vp^{*}\NN(T_W))^{\psi=0} \hookrightarrow (\BrigE)^{\psi=0}\otimes \Dcris(T_{W}).\]

\begin{rem}
More precisely, $\Mbar \in M_{2,2}(\HH(\GG_1))$, since $n'_{i}$ lie in $(1+\pi)\vp(\NN(T_W)) \subset (1+\pi)\vp(\BrigE)\otimes \Dcris(T_W)$.    
\end{rem}

This matrix $\Mbar$ will play a crucial role in the upcoming sections. 
More precisely, we will show that $\Mbar$ is a logarithmic matrix (in the sense of Sprung and Lei--Loeffler--Zerbes) that can be used in the decomposition of power series with unbounded denominators.

\subsection{Perrin-Riou's Exponental map, the $p$-adic regulator and the relation between them}

Recall that the eigenvalues of $\vp$ are not integral powers of $p$. Under this assumption, we can construct the \emph{Perrin-Riou's exponentia map} as (see \cite[Section 3.2.3]{PR1}, \cite[Section II.5]{BER2} for the details)
\begin{equation*}
    \Omega_{T_W,k-1}\colon \OO_{E}\otimes(\Brig)^{\psi=0}\otimes\Dcris(T_W) \to \HH_{E}(\Gamma)\otimes H^{1}_{\mathrm{Iw}}(\Qp,T_W).
\end{equation*}
Note that this map is a $\Lambda_{\Oe}(\GG)$-module homomorphism.
We can extend this to a $\Lambda_{E}(\GG)$-module homomorphism
\begin{equation*}
    \Omega_{W,k-1}\colon E\otimes(\Brig)^{\psi=0}\otimes\Dcris(W) \to \HH_{E}(\Gamma)\otimes\Hiw.
\end{equation*}
This map interpolates the Bloch--Kato exponential map 
$$\mathrm{exp}_{n,j} \colon \QQ_{p,n} \otimes \Dcris(W(j)) \to H^{1}(\QQ_{p,n}, W(j)),$$
where $j$ is any integer.

Recall we have fixed a topological generator $u$ of $1+p\Zp$ in Section \ref{setup}.
For $i\in \ZZ$, define $\ell_{i} = \dfrac{\log_{p}(1+X)}{\log_{p}(u)} - i$. 
It is easy to see that $\ell_{i} \in \HH_{E}(\GG)$ for all integers $i$. 
Berger gave a description of $\Omega_{W,k-1}$ in the terms of $\ell_i$'s (see \cite[Section II.1, Theorem II.13]{BER2} for more details) as follows:
\begin{equation}
    \Omega_{W,k-1}(z) = (\ell_{k-2}\circ\cdots\circ\ell_{0})(1-\vp)^{-1}(\bar{z}),
\end{equation}
where $z\in \HH_{E}(\GG)\otimes\Dcris(W)$ and $\bar{z}=(\mathfrak{M}\otimes 1)(z) $.

\begin{defi}[The $p$-adic regulator]
The Perrin-Riou $p$-adic regulator map $\LT$ for the $G_{\Qp}$-stable $\Oe$-lattice $T_W$ in $W$ is a $\Lambda_{\Oe}(\GG)$-homomorphism defined as
    \begin{equation*}
        \LT\coloneqq(\mathfrak{M}^{-1}\otimes1)\circ(1-\vp)\circ(h^{1}_{\mathrm{Iw}, T_W})^{-1} \colon H^{1}_{\mathrm{Iw}}(\Qp, T_W) \to \HH_{E}(\GG)\otimes \Dcris(T_W).
    \end{equation*}
\end{defi}

The $p$-adic regulator $\LT$ can be extended to a $\Lambda_{E}(\GG)$-homomorphism as
\begin{equation*}
    \LW\colon \Hiw \to \HH_{E}(\GG)\otimes \Dcris(W).
\end{equation*}

The $p$-adic regulator map $\LW$ and the exponential map $\Omega_{W,k-1}$ are related by the following lemma:
\begin{lem}
\label{lemma1}
    As maps on $\Hiw$, we have
    \begin{equation*}
        \LW = (\mathfrak{M}^{-1}\otimes 1)(\prod_{i=0}^{k-2}\ell_{i})(\Omega_{W,k-1})^{-1}.
    \end{equation*}

    In other words,
    \begin{equation}
        \Omega_{W,k-1}(\LW(z)) = (\prod_{i=0}^{k-2}\ell_{i})(z),
    \end{equation}
    for all $z\in \Hiw$. 
\end{lem}
\begin{proof}
    See \cite[Theorem 4.6]{LLZ2}.
\end{proof}

The following lemma gives a relationship between $\LT$ and the Coleman maps.
\begin{lem}[Lei--Loeffler--Zerbes {{\cite[Lemma 3.3]{LLZ2}}}]
\label{lem1.8}
    For $z\in \NN(T_W)^{\psi=\green{1}}$, we have
    \begin{equation*}
        (1-\vp)(z) =\begin{bmatrix} (1+\pi)\otimes w_{1} & (1+\pi)\otimes w_{2} \end{bmatrix}\Mbar\Col(z). 
    \end{equation*}
    Thus, we can rewrite $\LT$ in terms of $(1+\pi)\otimes w_{1}, (1+\pi)\otimes w_{2}$ as
\[\LT(z) = \begin{bmatrix} (1+\pi)\otimes w_{1} & (1+\pi)\otimes w_{2}\end{bmatrix}\Mbar\left(\Col\circ(h^{1}_{\mathrm{Iw},T_W})^{-1}\right)(z),\]
where $z \in H^{1}_{\mathrm{Iw}}(\Qp,T_W)$.
\end{lem}

\begin{proof}
    This can be proved using.
    $$(1-\vp)(z)= \Matwach\cdot\Col(z),$$
    and the definitions of $\Mbar$ and $\LT$.
\end{proof}

Now for $z\in \HiwT$, $\LT(z)$ is an element of $\HH_{E}(\GG)\otimes \Dcris(T_W)$. Hence, we can apply any character of $\GG$ to $\LT(z)$ to get an element in $E\otimes\Dcris(T_W)$. 
\begin{prop}
\label{prop1}
Let $z \in \HiwT$. Then for any integer $0 \leq i \leq k-2$, and for any Dirichlet character $\omega$ of conductor $p^n > 1$, we have
\begin{align}
    (1-\vp)^{-1}(1-p^{-1}\vp^{-1}) \chi^{i}(\LT(z)\otimes t^{i}e_{-i}) &\in \mathrm{Fil}^{0}(\Dcris(T_{W}(-i))),\\
    \vp^{-n}(\chi^{i}\omega(\LT(z)\otimes t^{i}e_{-i})) &\in \QQ_{p,n}\otimes \mathrm{Fil}^{0}\Dcris(T_{W}(-i)) \label{16},
\end{align}
where $\chi$ is the $p$-adic cyclotomic character and $e_{-i}$ is the basis of $\Dcris(\ZZ_{p}(-i))$. 
\end{prop}

\begin{proof}
    We replace $V$ with $T_W$ in \cite[Proposition 4.8]{LLZ2} and the result follows. 
\end{proof}
 
\begin{lem}
\label{lemma2}
We have
$$\det(\LW) = \prod_{i=0}^{k-2} \ell_{i}.$$    
\end{lem}

\begin{proof}
    Since the Hodge--Tate weights of $W$ are $0$ and $k-1$, we have
    $$\mathrm{dim}_{E}(\mathrm{Fil}^{i}\Dcris(W)) = 1,$$
    for all $-(k-2)\leq i \leq 0$. 
    Thus, replacing $V$ with $W$, putting $d=2$ and $n_{i}=1$ for all integers $0\leq i\leq(k-2)$ in \cite[Corollary 4.7]{LLZ2}, we get
    $$\det(\LW) = \prod_{i=0}^{k-2} (\ell_{i})^{2-1}.$$
\end{proof}


\subsection{The matrices of $\Omega_{W,k-2}, \LW$ and their connection with $\Mbar$}
For the rest of the article, fix an eigenbasis $\Beigen= \{w_{\alpha}, w_{\beta}\}$ of $\vp$ for $\Dcris(T_W)$, that is, $\Beigen$ is a basis for $\Dcris(T_W)$ and $\vp(w_{\alpha}) = (\alpha)^{-1}w_{\alpha}$ and $\vp(w_{\beta}) = (\beta)^{-1}w_{\beta}$. 
Thus, $(1+\pi)\otimes w_{\alpha}, (1+\pi)\otimes w_{\beta}$ is a basis for $(\BrigE)^{\psi=0}\otimes\Dcris(T_W)$.
We denote the basis $\{w_{1}, w_{2}\}$ for $\Dcris(T_W)$ by $\Bstd$, which we have defined in subsection 3.1.

Recall that the matrix of $\vp$ with respect to $\Bstd$ is
$$A_{\vp}=\MatA, $$
and thus we get
$$\begin{pmatrix} \alpha^{-1} & 0 \\[6pt] 0 & \beta^{-1} \end{pmatrix} = Q^{-1}A_{\vp}Q,$$
where $Q=\MatQ$.

We know that the $\Lambda_{\Oe}(\GG)$-rank of $\HiwT$ is 2, since $\dim_{E}\Dcris(W)=2$ and $T_W$ is an $\Oe$-lattice in $W$. See \cite[Section 3.2]{PR1} and \cite[Proposition 2.7]{BER2} for precise details. 
Thus, we may fix a $\Lambda_{\Oe}(\GG)$-basis $\{z_{1}, z_{2}\}$ for $\HiwT$.   

\subsubsection{The matrix of $\Omega_{T_W,k-1}$}

Using the expoential map $\Omega_{T_{W},k-1}$,
we obtain the following equations
\begin{align*}
    \Omega_{W,k-1}((1+\pi)\otimes w_{\alpha}) &= a\cdot z_{1} + c\cdot z_{2},\\
    \Omega_{W,k-1}((1+\pi)\otimes w_{\beta}) &= b\cdot z_{1} + d\cdot z_{2},
\end{align*}
where $a,b,c,d$ are elements in the distribution ring $\HH_{E}(\GG)$.
In other words, we can write these equations as
$$\begin{bmatrix} \Omega_{W,k-1}((1+\pi)\otimes w_{\alpha}) & \Omega_{W,k-1}((1+\pi)\otimes w_{\beta})\end{bmatrix} = \begin{bmatrix} z_{1} & z_{2}\end{bmatrix} \begin{pmatrix} a & b \\[6pt] c & d\end{pmatrix}.$$
Therefore, with respect to the basis $\{z_{1}, z_{2}\}$ for $\HiwT$ and the basis $\Beigen$ for $\Dcris(T_W)$ we can describe the matrix $M_{\Omega}$ of $\Omega_{T_W,k-1}$ as, 
\begin{equation}
    M_{\Omega} = \begin{pmatrix} a & b\\[6pt] c& d \end{pmatrix}.
\end{equation}

Recall that, if $F\in \HH_{E,r}(\GG)$, we say $F$ is $O(\log_{p}^r)$.

\begin{lem}
\label{lemma3}
    The elements $a,c$ are $O(\log_{p}^{v_{p}(\beta)})$, whereas $b,d$ are $O(\log_{p}^{v_{p}(\alpha)})$.  
\end{lem}

\begin{proof}
From \cite[Section 3.2.4]{PR1}, we note that for any $\Oe$-lattice $T$ in a crystalline representation $V$, if $f$ is an element of $(\BrigE)^{\psi=0}\otimes \mathbb{D}_{\nu}(T(j))$, where $\mathbb{D}_{\nu}(T(j))$ is a subspace of $\Dcris(T(j))$ in which $\vp$ has slope $\nu$, then $\Omega_{T(j), h+j}(f)$ is $O(\log_{p}^{h+\nu})$ and $h\in\ZZ_{>0}$ such that $\mathrm{Fil}^{-h}(\Dcris(T))=\Dcris(T)$. 
In other words, $\Omega_{T(j),h+j}(f)$ lies in $\HH_{h+\nu}(\GG)\otimes H^{1}_{\mathrm{Iw}}(\Qp, T(j))$.

For the crystalline representation $W$ and the lattice $T_W$ in $W$, we know that $\vp(w_{\alpha})=\alpha^{-1}w_{\alpha}$ and $\vp(w_\beta)=\beta^{-1}w_{\beta}$. Thus, $\Omega_{T_{W}, k-1}((1+\pi)\otimes w_{\alpha})$ is $O(\log^{(k-1)-v_{p}(\alpha)}_{p}) = O(\log^{v_{p}(\beta)}_{p})$, since $v_{p}(\alpha) + v_{p}(\beta)=k-1$. 
Similarly, $\Omega_{T_{W}, k-1}((1+\pi)\otimes w_{\beta})$ is $O(\log^{v_{p}(\alpha)}_{p})$.

But $\Omega_{T_{W}, k-1}((1+\pi)\otimes w_{\alpha}) = a\cdot z_{1} + c\cdot z_2 \in \HH_{E,v_{p}(\beta)}(\GG)\otimes \HiwT$. 
Therefore we conclude that $a$ and $c$ have growth $O(\log_{p}^{v_{p}(\beta)})$.
In the same manner, $b$ and $d$ have growth $O(\log_{p}^{v_{p}(\alpha)})$.
\end{proof}

\subsubsection{The matrix of the $p$-adic regulator $\LT$}

After applying the $p$-adic regulator $\LT$ on the $\Lambda_{\Oe}(\GG)$-basis $z_{1}, z_{2}$ of $\HiwT$, we get 
\begin{align*}
    \LT(z_{1}) &= x_{1}\cdot w_{\alpha} + x_{3}\cdot w_{\beta},\\
    \LT(z_{1}) &= x_{2}\cdot w_{\alpha} + x_{4}\cdot w_{\beta}.
\end{align*}
We can rewrite these equations as
\begin{equation}
    \begin{bmatrix} \LT(z_{1}) & \LT(z_{2})\end{bmatrix} = \begin{bmatrix} w_{\alpha} & w_{\beta}\end{bmatrix} \begin{pmatrix} x_1 & x_2 \\[6pt] x_3 & x_4\end{pmatrix}.
\end{equation}

Hence, using the basis $\{z_{1}, z_{2}\}$ for $\HiwT$ and the basis $\Beigen$ for $\Dcris(T_W)$, we get a matrix $[\LT]_{\Beigen} \in M_{2,2}(\HH_{E}(\GG))$ of $\LT$ as
\begin{equation}
    [\LT]_{\Beigen} = \begin{pmatrix} x_1 & x_2 \\[6pt] x_3 & x_4\end{pmatrix} .\label{19}
\end{equation}

\begin{lem}
\label{lem1.12}
We have the equation
\begin{equation}
        [\LT]_{\Beigen} = \adj M_{\Omega},
    \end{equation}
    where $\adj M_{\Omega}$ is the adjugate matrix of $M_\Omega$.
    In particular, $x_{1},x_{2}$ are $O(\log_{p}^{v_{p}(\alpha)})$ and $x_{3},x_{4}$ are $O(\log_{p}^{v_{p}(\beta)})$.
\end{lem}

\begin{proof}

From Lemma \ref{lemma1}, we know that
$$\OW\circ\LW = \left(\prod_{i=0}^{k-2} \ell_{i}\right).$$
By restricting to the $\Oe$-lattice $T_W$ in $W$, we have, for any $z \in \HiwT$,
\begin{align*}
    \OT(\LT(z)) = \left(\prod_{i=0}^{k-2} \ell_{i}\right)(z).
\end{align*}
Thus, by using the $\Lambda_{\Oe}(\GG)$-basis $\{z_1, z_{2}\}$ for $\HiwT$, we get
  \begin{align}
       \OT(\LT(z_{1})) &= \left(\prod_{i=0}^{k-2} \ell_{i}\right)(z_{1}),\label{eq1}\\[6pt]
        \OT(\LT(z_{2})) &= \left(\prod_{i=0}^{k-2} \ell_{i}\right)(z_{2}).\label{eq2}
  \end{align}
In matrix form, we can rewrite the equations \eqref{eq1} and \eqref{eq2}
  \begin{equation*}
      M_{\Omega}[\LT]_{\Beigen}= \left(\prod_{i=0}^{k-2} \ell_{i}\right) I_{2}. 
  \end{equation*}
  From Lemma \ref{lemma2}, we have $\det(\LT) = \prod_{i=0}^{k-2}\ell_{i}$, hence we have
  $\det([\LT]_{\Beigen})= \prod_{i=0}^{k-2}\ell_{i}.$ 
  Thus
  $$M_{\Omega}[\LT]_{\Beigen} = \det(\LT) I_{2}.$$

  Hence, 
  \begin{equation}
      \begin{pmatrix} x_1 & x_2 \\[6pt] x_3 & x_4\end{pmatrix}=[\LT]_{\Beigen} = \adj M_{\Omega} = \begin{pmatrix} d & -b \\[6pt] -c & a \end{pmatrix}, 
  \end{equation}
  where $\adj M_{\Omega}$ is the adjugate matrix of the matrix $M_{\Omega}$.
  Thus, from Lemma \ref{lemma3}, we get $x_{1}, x_{2}$ have growth $O(\log_{p}^{v_{p}(\alpha)})$ and $x_{3},x_{4}$ have growth $O(\log_{p}^{v_{p}(\beta)})$. 
\end{proof}

We use the basis $\Bstd$ of $\Dcris(T_W)$ and the basis $\{z_{1}, z_{2}\}$ for $\HiwT$ to get another matrix $[\LT]_{\Bstd}$ such that
$$\begin{bmatrix} \LT(z_{1}) & \LT(z_{2}) \end{bmatrix} = \begin{bmatrix} w_{1} & w_{2}\end{bmatrix} [\LT]_{\Bstd}.$$
Since $\begin{bmatrix} w_{1} & w_{2}\end{bmatrix} = \begin{bmatrix} w_{\alpha} & w_{\beta}\end{bmatrix}Q^{-1}$, we have
\begin{equation}
    [\LT]_{\Beigen} = Q^{-1}[\LT]_{\Bstd}. 
\end{equation}

For any non-negative integer $n$, we write $\omega_{n}(1+X) = (1+X)^{p^n}-1$. 
Let $\Phi_{n}(1+X)=\omega_{n}(1+X)/\omega_{n-1}(1+X)$ be the $p^{n}$-th cyclotomic polynomial for integers $n>1$.
Recall from Section 2, topological generators $\gamma_{0}$ of $\GG_{1}$ and $u$ of $1+p\Zp$ such that $\chi(\gamma_{0})=u$, where $\chi$ is the $p$-adic cyclotomic character. 
For any integer $m\geq 1$, we define
\begin{align*}
    \Phi_{n,m}(1+X) &= \prod_{j=0}^{m-1}\Phi_{n}(u^{-j}(1+X)),\\
    \omega_{n,m}(1+X) &= \prod_{j=0}^{m-1} \omega_{n}(u^{-j}(1+X)),\\
    \delta_{m}(X) &= \prod_{j=0}^{m-1}(u^{-j}(1+X)-1),\\
    \log_{p,m}(1+X) &= \prod_{j=0}^{m-1}\log_{p}(u^{-j}(1+X)). 
 \end{align*}
    
\begin{rem}
    \label{logproplem}
    Note that the zeros of  $\log_{p,m}(1+X)$, all having simple order, are of the form $u^{j}\zeta - 1$, where $\zeta$ is a $p^{n}$-th root of unity, for all $n\geq 1$ and for all $0\leq j\leq m-1$.
\end{rem}

Recall from Proposition \ref{prop1}, for any Dirichlet character $\omega$ of conductor $p^{n}, n>1$, we have
$$\vp^{-n}(\chi^{i}\omega(\LT(z) \otimes t^{i}e_{-i})) \in \QQ_{p,n} \otimes \mathrm{Fil}^{0}\Dcris(T_{W}(-i)),$$
for any $z\in \HiwT$ and $0 \leq i \leq (k-2)$.

\begin{prop}
    \label{prop2}
    The second row of the matrix $[\vp^{-n}\LT]_{\Bstd}$ is divisible by $\Phi_{n-1,k-1}(\gamma_{0})$ over $\HH_{E}(\Gamma)$, for all integers $n>1$.
\end{prop}

\begin{proof}
    The Hodge--Tate weights of the crystalline representation $W$ are $0$ and $k-1$. Thus, for $0 \leq i \leq k-2$, we have
    $$\dim_{E}(\mathrm{Fil}^{0}\Dcris(W(-i))) = \dim_{E}(\mathrm{Fil}^{-i}\Dcris(W))=1.$$
    We know that $\mathrm{Fil}^{-i}\Dcris(T_W)$ is a rank one free $\Oe$-submodule of $\Dcris(T_W)$ generated by $w_{1}$ for all $0\leq i \leq k-2$. 
    Thus, $\mathrm{Fil}^{0}\Dcris(T_{W}(-i))$ is generated by $w_{1}\otimes t^{i}e_{-i}$. 

    Write
    \begin{equation}
        \vp^{-n}(\LT(z)\otimes t^{i}e_{-i}) = F_{1,z}\cdot (w_{1}\otimes p^{-ni}t^{i}e_{-i}) + F_{2,z}\cdot (w_{2}\otimes p^{-ni}t^{i}e_{-i}),
    \end{equation}
    where $F_{1,z}, F_{2,z}\in \HH_{E}(\GG)$ and $z\in \HiwT$.

    Thus, \eqref{16} implies 
    \begin{equation}
        F_{2,z}(\chi^{i}\omega)=0, \label{25}
    \end{equation}
    for all $0 \leq i \leq k-2$ and for all Dirichlet character $\omega$ of conductor $p^n$, where $n>1$. Then \cite[Theorem 5.4]{LLZ1} implies $\Phi_{n-1,k-1}(\gamma_{0})$ divides $F_{2,z}$. 

     Using the basis $z_{1}, z_{2}$ for $\HiwT$ and the basis $\Bstd$ for $\Dcris(T_W)$, we get 
    \begin{equation*}
        \begin{bmatrix} \vp^{-n}(\LT(z_{1}) \otimes t^{i}e_{-i}) & \vp^{-n}(\LT(z_{2}) \otimes t^{i}e_{-i})\end{bmatrix} = \begin{bmatrix} w_{1}\otimes t^{i}e_{-i} & w_{2}\otimes t^{i}e_{-i}\end{bmatrix} \begin{pmatrix} F_{1,z_{1}} & F_{1,z_{2}} \\[6pt] F_{2,z_{1}} & F_{2,z_{2}} \end{pmatrix}.
    \end{equation*}
    Then the matrix of $\vp^{-n}\LT$ with respect to basis $\Bstd$ is
    $$[\vp^{-n}\LT]_{\Bstd} = \begin{pmatrix} F_{1,z_{1}} & F_{1,z_{2}} \\[6pt] F_{2,z_{1}} & F_{2,z_{2}} \end{pmatrix}.$$
    Note that $[\vp^{-n}\LT]_{\Bstd} \in M_{2,2}(\HH(\GG))$. 
    Hence, using \eqref{25}, we deduce that $\Phi_{n-1,k-1}(\gamma_{0})$ divides both $F_{2,z_1}$ and $F_{2,z_2}$
\end{proof}

For $n>1$, we can write
\begin{align*}
    \begin{bmatrix} \vp^{-n}(\LT(z_{1})) & \vp^{-n}(\LT(z_{2}))\end{bmatrix} &= \begin{bmatrix} \vp^{-n}(w_{\alpha}) & \vp^{-n}(w_{\beta}) \end{bmatrix}[\LT]_{\Beigen},\\
    &= \begin{bmatrix} w_{\alpha} & w_{\beta}\end{bmatrix}\begin{pmatrix} \alpha^{n} & 0 \\[6pt] 0 & \beta^{n}\end{pmatrix}[\LT]_{\Beigen}.
\end{align*}
Thus, we get a matrix for $\vp^{-n}\LT$ with respect to the eigenbasis $\Beigen$ for $\Dcris(T_W)$ and using \eqref{19}, we get
\begin{equation}
    [\vp^{-n}\LT]_{\Beigen} = \begin{pmatrix} \alpha^{n} & 0 \\[6pt] 0 & \beta^{n}\end{pmatrix}[\LT]_{\Beigen} = \begin{pmatrix} \alpha^{n} x_{1} & \alpha^{n} x_{2} \\[6pt] \beta^{n} x_{3} & \beta^{n} x_{4}\end{pmatrix}. \label{29}
\end{equation}

\section{Logarithmic matrix $\Mbar$ and the factorization of power series in one variable}
\label{sec5}
In this section, we will first explore some properties of $\Mbar$ which imply that $\Mbar$ is a logarithmic matrix in the sense of Sprung and Lei--Loeffler--Zerbes. 
Next, we will use $\Mbar$ to decompose power series with certain growth conditions into power series with bounded coefficients.  

\subsection{Properties of $\Mbar$}
For any $z\in \HiwT$, we have
$$\LT(z) = \begin{bmatrix} (1+\pi)\otimes w_{1} & (1+\pi)\otimes w_{2} \end{bmatrix} \Mbar \begin{pmatrix} \Col_{1}(z) \\[6pt] \Col_{2}(z)\end{pmatrix}.$$

In matrix form, we write
$$\begin{bmatrix}\LT(z_1) & \LT(z_2)\end{bmatrix} = \begin{bmatrix} (1+\pi)\otimes w_{1} & (1+\pi)\otimes w_{2} \end{bmatrix} \Mbar \begin{pmatrix} \Col_{1}(z_{1}) & \Col_{1}(z_{2}) \\[6pt] \Col_{2}(z_{1}) & \Col_{2}(z_{2})\end{pmatrix}.$$
Thus, 
\begin{equation}
    [\LT]_{\Bstd} = \Mbar\begin{pmatrix} \Col_{1}(z_{1}) & \Col_{1}(z_{2}) \\[6pt] \Col_{2}(z_{1}) & \Col_{2}(z_{2})\end{pmatrix}. 
\end{equation}
Similarly, we have
\begin{equation}
    [\LT]_{\Beigen} = Q^{-1}\Mbar\begin{pmatrix} \Col_{1}(z_{1}) & \Col_{1}(z_{2}) \\[6pt] \Col_{2}(z_{1}) & \Col_{2}(z_{2})\end{pmatrix}, \label{ltbeigen}
\end{equation}
since $[\LT]_{\Beigen} = Q^{-1}[\LT]_{\Bstd}.$

\begin{prop}
    \label{prop4}
    The elements in the first row of $Q^{-1}\Mbar$ are inside $\HH_{E,v_{p}(\alpha)}(\GG_{1})$, while the elements in the second row are in the $\HH_{E,v_{p}(\beta)}(\GG_{1})$. 
\end{prop}

\begin{proof}
    Recall from Lemma \ref{lem1.12},  
    $$[\LT]_{\Beigen}=\adj M_{\Omega} = \begin{pmatrix} x_{1} & x_{2} \\[6pt] x_{3} & x_{4} \end{pmatrix}.$$
    Therefore, \eqref{ltbeigen} implies
    \begin{equation*}
        \begin{pmatrix} x_{1} & x_{2} \\[6pt] x_{3} & x_{4} \end{pmatrix} = Q^{-1}\Mbar\begin{pmatrix} \Col_{1}(z_{1}) & \Col_{1}(z_{2}) \\[6pt] \Col_{2}(z_{1}) & \Col_{2}(Z_{2})\end{pmatrix}.\label{mbar1}
    \end{equation*}
    But $\Col_{i}(z_{j})$ are $O(1)$ for $i,j\in\{1,2\}$, since they lie in the Iwasawa algebra $\Lambda_{E}(\GG)$. 
    Therefore, after writing $Q^{-1}\Mbar = \begin{pmatrix} P_{1} & P_{2} \\[6pt] P_{3} & P_{4}\end{pmatrix}$, we get
    \begin{equation}
        \begin{pmatrix} x_{1} & x_{2} \\[6pt] x_{3} & x_{4} \end{pmatrix} = \begin{pmatrix} P_{1} & P_{2} \\[6pt] P_{3} & P_{4}\end{pmatrix} \begin{pmatrix} \Col_{1}(z_{1}) & \Col_{1}(z_{2}) \\[6pt] \Col_{2}(z_{1}) & \Col_{2}(Z_{2})\end{pmatrix}.\label{mbar1}
    \end{equation}
   
    We take isotypic components on both sides with respect to the trivial character of $\Delta$.
    After writing $[\mathrm{Col}^{\Delta}]$ for $\begin{pmatrix} \Col^{\Delta}_{1}(z_{1}) & \Col^{\Delta}_{1}(z_{2}) \\[6pt] \Col^{\Delta}_{2}(z_{1}) & \Col^{\Delta}_{2}(Z_{2})\end{pmatrix}$, equation \eqref{mbar1} becomes 
    \begin{equation}
        \begin{pmatrix} x_{1}^{\Delta} & x_{2}^{\Delta} \\[6pt] x_{3}^{\Delta} & x_{4}^{\Delta} \end{pmatrix} \adj [\mathrm{Col}^{\Delta}] = \begin{pmatrix} P_{1} & P_{2} \\[6pt] P_{3} & P_{4}\end{pmatrix} \det([\mathrm{Col}^{\Delta}]).
    \end{equation}
    The result follows from the Lemma \ref{lem1.12}, since $\det([\mathrm{Col}^{\Delta}])$ is again $O(1)$ and $x_{1}^{\Delta}, x_{2}^{\Delta} \in \HH_{E,v_{p}(\alpha)}(\GG_{1})$ and $x_{3}^{\Delta}, x_{4}^{\Delta} \in \HH_{E,v_{p}(\beta)}(\GG_{1})$.    
\end{proof}

\begin{lem}
\label{lem5}
    The second row of $A_{\vp}^{-n}\Mbar$ is divisible by the cyclotomic polynomial $\Phi_{n-1,k-1}(\gamma_{0})$ over $\HH_{E}(\GG_{1})$. 
\end{lem}

\begin{proof}
We know that
$$\begin{bmatrix}\LT(z_1) & \LT(z_2)\end{bmatrix} = \begin{bmatrix} (1+\pi)\otimes w_{1} & (1+\pi)\otimes w_{2} \end{bmatrix} \Mbar \begin{pmatrix} \Col_{1}(z_{1}) & \Col_{1}(z_{2}) \\[6pt] \Col_{2}(z_{1}) & \Col_{2}(z_{2})\end{pmatrix}.$$
Let us denote $\begin{pmatrix} \Col_{1}(z_{1}) & \Col_{1}(z_{2}) \\[6pt] \Col_{2}(z_{1}) & \Col_{2}(z_{2})\end{pmatrix}$ by $[\mathrm{Col}]$.
By an abuse of notation, we write $\vp$ for $1\otimes \vp$. Here $1\otimes\vp$ mean $\vp$ does not act on $(1+\pi)$ and acts as $\vp$ on $w_i$. 

After applying $\vp^{-n}$ on both sides of the above equation, we obtain
\begin{align*}
    \vp^{-n}\left(\begin{bmatrix} \LT(z_1) & \LT(z_2)\end{bmatrix}\right) &=\begin{bmatrix}\vp^{-n}(\LT(z_1)) & \vp^{-n}(\LT(z_2))\end{bmatrix},\\
    &= \begin{bmatrix} (1+\pi)\otimes \vp^{-n}(w_{1}) & (1+\pi)\otimes \vp^{-n}(w_{2}) \end{bmatrix} \Mbar[\mathrm{Col}],\\ 
     &= \begin{bmatrix} (1+\pi)\otimes w_{1} & (1+\pi)\otimes w_{2} \end{bmatrix} A_{\vp}^{-n}\Mbar[\mathrm{Col}].
\end{align*}
Therefore, we get
\begin{equation*}
    [\vp^{-n}\LT]_{\Bstd}=A^{-n}_{\vp}\Mbar[\mathrm{Col}].
\end{equation*}
After rearranging the above equation, we get
\begin{equation}
    [\vp^{-n}\LT]_{\Bstd} \adj [\mathrm{Col}] = A^{-n}_{\vp}\Mbar\det([\mathrm{Col}]),
\end{equation}
From \cite[Proposition 4.11, Theorem 4.12, Corollary 4.15]{LLZ2}, we can conclude that $\Phi_{n-1,k-1}(\gamma_{0})$ does not divide $\det([\mathrm{Col}])$.
Thus, the result follows from Proposition \ref{prop2}.
\end{proof}

\begin{lem}
\label{lem1.16}
    The determinant of matrix $\Mbar$ is $\dfrac{\log_{p,k-1}(\gamma_{0})}{\delta_{k-1}(\gamma_{0}-1)}$ 
     upto a unit in $\Lambda_{E}(\GG_{1})$.
\end{lem}

\begin{proof}
    This is \cite[Corollary 3.2]{LLZ2}. See also \cite[Lemma 2.7]{BL1}.
\end{proof}

Thus, Proposition \ref{prop4}, Lemma \ref{lem5}, and Lemma \ref{lem1.16} imply that the matrix $\Mbar$ is a logarithmic matrix in the sense of Sprung and Lei--Loeffler--Zerbes.

\subsection{Factorization using $\Mbar$}
Let $F,G$ be power series in $\HH_{E}(\GG)$. 
We write $F\sim G$, if $F$ is $O(G)$ and $G$ is $O(F)$.
Recall that, for power series $F$and $G$, we say $F$ is $O(G)$ if $\lv F\lv_{\rho} = O(\lv G \lv_{\rho})$ as $\rho \to 1^{-}$.

\begin{lem}
\label{5.4}
    We have $\det(\Mbar) \sim \log_{p}^{k-1}$.
\end{lem}
\begin{proof}
    From Lemma \ref{lem1.16}, we get $\det(\Mbar)=*\dfrac{\log_{p,k-1}(\gamma_0)}{\delta_{k-1}(\gamma_{0}-1)}$, where $*$ is a unit in $\Lambda_{E}(\GG)$. 
    Hence the result follows from the definition of $\log^{k-1}_{p}(\gamma_0)$ and the fact that $\delta_{k-1}(\gamma_{0}-1)$ is polynomial and hence $O(1)$.
\end{proof}

\begin{thm}
    \label{thm1}
   For $\lambda\in\{\alpha, \beta\}$, let $F_{\lambda} \in \HH_{E,v_{p}(\lambda)}(\GG)$,  
    such that for any integer $0\leq j \leq k-2$ and for any Dirichlet character $\omega$ of conductor $p^n$ we have $F_{\lambda}(\chi^{j}\omega)=\lambda^{-n}C_{\omega,j}$, where $C_{\omega,j}\in \overline{\Qp}$ that is independent of $\lambda$. 
    Then, there exist $F_{\flat}, F_{\sharp} \in \Lambda_{E}(\GG)$ such that
    \begin{equation}
        \begin{pmatrix} F_{\alpha} \\[6pt] F_{\beta} \end{pmatrix} = Q^{-1}\Mbar \begin{pmatrix} F_{\sharp} \\[6pt] F_{\flat} \end{pmatrix}.
    \end{equation}
    
\end{thm}

\begin{proof}
    From Lemma \ref{lem5}, we know that the second row of $A_{\vp}^{-n}\Mbar$ is divisible by $\Phi_{n-1,k-1}(\gamma_{0})$. 
Hence we can write 
$$A_{\vp}^{-n}\Mbar = \begin{pmatrix} a & b \\[6pt] c\cdot\Phi_{n-1,k-1}(\gamma_{0}) & d\cdot\Phi_{n-1,k-1}(\gamma_{0})  \end{pmatrix},$$
where $a,b,c,d$ are power series.

Recall,
$Q^{-1}= \frac{1}{\det Q}\begin{pmatrix} vp^{k-1} & \beta \\[6pt] vp^{k-1} & \alpha\end{pmatrix},$ and
$\adj Q^{-1}\Mbar = \begin{pmatrix} P_4 & -P_2 \\[6pt] -P_3 & P_1\end{pmatrix}$. 
Note that $\det Q\neq 0$ since $\alpha\neq\beta$.

Thus for every positive integer $n$, we have
\begin{align*}
    Q^{-1}\Mbar &= Q^{-1}A^{n}_{\vp}QQ^{-1}A_{\vp}^{-n}\Mbar,\\ 
                &=\begin{pmatrix}  \frac{1}{\alpha^n} & 0 \\[12pt] 0 & \frac{1}{\beta^n} \end{pmatrix}  \dfrac{1}{\det Q}\begin{pmatrix} vp^{k-1} & \beta \\[12pt] vp^{k-1} & \alpha\end{pmatrix} \begin{pmatrix} a & b \\[12pt] c\cdot\Phi_{n-1,k-1}(\gamma_{0}) & d\cdot\Phi_{n-1,k-1}(\gamma_{0})  \end{pmatrix}  , \\[6pt]
                &= \frac{1}{\det Q}\begin{pmatrix}\frac{1}{\alpha^n} & 0 \\[12pt] 0 & \frac{1}{\beta^n}\end{pmatrix} \begin{pmatrix}a\cdot vp^{k-1}+\beta\cdot c\cdot\Phi_{n-1,k-1}(\gamma_{0}) & b\cdot vp^{k-1}+\beta\cdot d\cdot\Phi_{n-1,k-1}(\gamma_{0}) \\[12pt] a\cdot vp^{k-1}+\alpha\cdot c\cdot\Phi_{n-1,k-1}(\gamma_{0}) & b\cdot vp^{k-1}+\alpha\cdot d\cdot\Phi_{n-1,k-1}(\gamma_{0})\end{pmatrix},\\[6pt]
                &= \frac{1}{\det Q}\begin{pmatrix}\frac{1}{\alpha^{n}}(a\cdot vp^{k-1}+\beta\cdot c\cdot\Phi_{n-1,k-1}(\gamma_{0})) & \frac{1}{\alpha^{n}}(b\cdot vp^{k-1}+\beta\cdot d\cdot\Phi_{n-1,k-1}(\gamma_{0})) \\[12pt] \frac{1}{\beta^{n}}(a\cdot vp^{k-1}+\alpha\cdot c\cdot\Phi_{n-1,k-1}(\gamma_{0})) & \frac{1}{\beta^{n}}(b\cdot vp^{k-1}+\alpha\cdot d\cdot\Phi_{n-1,k-1}(\gamma_{0}))\end{pmatrix}.
\end{align*}

Hence
\begin{equation}
    \adj Q^{-1}\Mbar = \frac{1} {\det Q} \begin{pmatrix} \dfrac{1}{\beta^{n}}(b\cdot vp^{k-1}+\alpha\cdot d\cdot\Phi_{n-1,k-1}(\gamma_{0})) & \dfrac{-1}{\alpha^{n}}(b\cdot vp^{k-1}+\beta\cdot d\cdot\Phi_{n-1,k-1}(\gamma_{0})) \\[12pt] \dfrac{-1}{\beta^{n}}(a\cdot vp^{k-1}+\alpha\cdot c\cdot\Phi_{n-1,k-1}(\gamma_{0})) & \dfrac{1}{\alpha^{n}}(a\cdot vp^{k-1}+\beta\cdot c\cdot\Phi_{n-1,k-1}(\gamma_{0})) \end{pmatrix}. \label{35}
\end{equation}

For integer $0 \leq j \leq (k-2)$, and Dirichlet character $\omega$ of conductor $p^n$, equation \eqref{35} implies
\begin{equation*}
    (\adj Q^{-1}\Mbar)(\chi^{j}\omega)=\begin{pmatrix} \dfrac{1}{\beta^n}* & \dfrac{-1}{\alpha^n}* \\[12pt] \dfrac{-1}{\beta^n}*' & \dfrac{1}{\alpha^n}*'  \end{pmatrix},
\end{equation*}
where $*,*'$ are in $\overline{\Qp}$.

Thus, if we write $F_{1}= P_{4}F_{\alpha} - P_{2}F_{\beta},$ 
then, from \eqref{35}, we get
\begin{align*}
    F_{1}(\chi^{j}\omega) &= P_{4}(\chi^{j}\omega)F_{\alpha}(\chi^{j}\omega) - 
                            P_{2}(\chi^{j}\omega)F_{\beta}(\chi^{j}\omega),\\[6pt]
                          &= \frac{1}{\beta^n}\frac{* C_{\omega,j}}{\alpha^n} - \frac{1}{\alpha^n}\frac{* C_{\omega,j}}{\beta^n},\\[6pt]
                          &=0.
\end{align*}

Similarly, if we write $F_{2}= -P_{3}F_{\alpha} + P_{1}F_{\beta}$, then $F_{2}(\chi^{j}\omega)=0$.

Hence, for every positive integer $n$, the zeros of $\Phi_{n-1,k-1}$ are also zeros of $F_{1}$ and $F_{2}$. 
In other words, the roots of $\det ( Q^{-1}\Mbar)= (\text{some constant})\dfrac{\log_{p,k-1}(\gamma_0)}{\delta_{k-1}(\gamma_{0}-1)}$ are also the roots of $F_{1}$, and $F_{2}$.  
Therefore, $\det Q^{-1}\Mbar$ divides both $F_{1},F_{2}$ in $\HH_{E}(\GG)$.

Note that $P_{4}F_{\alpha}$ is $O(\log_{p}^{k-1})$, since $P_{4}$ is $O(\log_{p}^{v_{p}(\beta)})$ and $F_{\alpha}$ is $O(\log_{p}^{v_{p}(\alpha)})$. Similarly, $P_{2}F_{\beta}, P_{3}F_{\alpha},$ and $P_{1}F_{\beta} $ are $O(\log_{p}^{k-1})$. 

Let 
\begin{equation}
\begin{matrix}
   F_{\sharp} =\dfrac{P_{4}F_{\alpha} - P_{2}F_{\beta}}{\det Q^{-1}\Mbar} & \text{and} & F_{\flat} = \dfrac{-P_{3}F_{\alpha} + P_{1}F_{\beta}}{\det Q^{-1}\Mbar}.
   \end{matrix}\label{37}
\end{equation}
Then $F_{\sharp}$ and $F_{\flat}$ have bounded coefficients since the numerators of both of them are $O(\log_{p}^{k-1})$, denominators of both of them are $\det(Q^{-1}\Mbar)$ and by Lemma \ref{5.4} $\det(Q^{-1}\Mbar)\sim \log_{p}^{k-1}$. Hence $F_{\sharp}$ and $F_{\flat}$ are $O(1)$ (i.e. bounded). Therefore, we can conclude that $F_{\sharp}$ and $F_{\flat}$ are in $\Lambda_{E}(\GG)$. This completes the proof.
\end{proof} 

\section{Preliminaries about ray class groups, Hecke characters, and the two variable distribution algebras}
For the rest of the article, we fix a quadratic imaginary field $K$ and a prime $p\geq 3$ which splits in $K$ as $p\mathcal{O}_{K}=\PP\PPP$.
Let $h_{K}\geq 1$ be the class number of $K$.

\begin{assum}
    For the rest of the article, we assume that $p\not\mid h_{K}$.
\end{assum} 

\subsection{Ray class groups and ray class fields}
Let $K_{\infty}$ be the unique $\mathbb{Z}^{2}_{p}$ extension 
of $K$.
If $\mathcal{I}$ is an ideal of $K$, we write $G_{\mathcal{I}}$ for the ray class group $K$ modulo $\mathcal{I}$.
We define
\begin{equation*}
    \text{Cl}(K,p^{\infty})=G_{p^{\infty}} = \varprojlim_{n} G_{(p)^n}, \hspace{0.05in} G_{\mathfrak{p}^{\infty}} = \varprojlim_{n} G_{\mathfrak{p}^n}, \hspace{0.05in} G_{\overline{\mathfrak{p}}^{\infty}} = \varprojlim_{n} G_{\overline{\mathfrak{p}}^n}.
\end{equation*}
These are the Galois groups of the ray class fields $K(p^{\infty}), K(\mathfrak{p}^{\infty})$ and $K(\overline{\mathfrak{p}}^\infty)$ respectively.
For $\qqq\in\pset$, let $H(\qqq^{\infty})$ be the subfield of $G_{\qqq^{\infty}}$ such that $\Gal(H(\qqq^{\infty})/K)\cong \Zp$. 
Note $H(\qq^{\infty})\subset K_{\infty}$.

\begin{rem}
    We have an isomorphism $G_{p^{\infty}}\cong \Delta_{K} \times \Zp\times\Zp \cong \Delta \times \overline{\langle \gamma_{\PP}\rangle}\times \overline{\langle \gamma_{\PPP}\rangle}$, where $\Delta_{K}$ is a finite abelian group, $\gamma_{\PP}$ and $\gamma_{\PPP}$ topologically generate $\Zp$ parts of $G_{\PP^{\infty}}$ and $G_{\PPP^{\infty}}$ respectively. 
    Here by $\overline{\langle x \rangle}$ we mean the topological closure of the cyclic group generated by $x$.
\end{rem}

\begin{rem}
\label{rem5.1}
 By the assumption $p\not\mid h_{K}$, there exists a unique prime in $K_{\infty}$ above $\PP$ and a unique prime above $\PPP$. By an abuse of notation, we will also denote by $\PP$ and $\PPP$ by the unique prime above $\PP$ and $\PPP$ respectively in $K_{\infty}$. 
 Therefore, for $\qq\in\pset$, $\Gal(H(\qqq^{\infty})_{\qqq}/K_{\qqq})=\overline{\langle \gamma_{\qqq} \rangle}\cong \Gal(H(\qqq^{\infty})/K) \cong \Zp$.
\end{rem}

Since $p$ splits in $K$, the local field $K_{\qq}$ is isomorphic to $\Qp$, for $\qq\in\pset$. 
Thus,$1+\pi_{\qqq}\OO_{K_{\qqq}}\cong 1 + p\Zp$, where $\pi_{\qq}$ is a uniformizer of $\OO_{K_{\qqq}}$.
Recall the topological generator $u\in 1 +p\Zp$ such that $\chi(\gamma_{0})=u$, where $\gamma_0$ generates $\Gamma_{1}\cong \Zp$, and $\chi$ is $p$-adic cyclotomic character. 
Thus, we may set $u_{\PP}=u_{\PPP}=u$, where $u_{\qqq}$ is a topological generator of $1+\pi_{\qqq}\OO_{K_{\qqq}}$. 
From now on, we fix this $u$.


By local class field theory, there exists a group isomorphism (\emph{Artin map}) \[\Art_{\qqq}:\OO_{K_{\qqq}}^{\times}\to\Gal(H(\qqq^{\infty})_{\qqq}/K_{\qqq})\cong\Gal(H(\qqq^{\infty})/K),\]
such that
$$\Art_{\qqq}(u_{\qqq})=\Art_{\qqq}(u)=\gamma_{\qqq},$$
where $\gamma_{\qqq}$ is a topological generator of $\Gal(H(\qqq^{\infty})_{\qqq}/K_{\qqq})$.
By an abuse of notations, let $\gamma_{\qqq}$ be a topological generator of $\Gal(H(\qqq^{\infty})/K)$ and $\Art_{\qqq}(u)=\gamma_{\qqq}$.

\subsection{Hecke characters as the characters on the ray class groups}
Let $\ideles$ denote the group of ideles of $K$ and write $\ideles=\mathbb{A}_{\infty}^{\times}\times \mathbb{A}_{f}^{\times}$, where $\mathbb{A}_{\infty}^{\times}$ is the infinite part and $\mathbb{A}_{f}^{\times}$ is the finite part.
We can embed $K^{\times}$ into $\ideles$ diagonally.
Fix embeddings $i_{\infty}:\overline{\QQ} \hookrightarrow \CC$ and $i_{p}: \overline{\QQ}\hookrightarrow\CC_{p}$.
\begin{defi}[Hecke characters]
\hfill
\begin{enumerate}
        \item A \emph{Hecke character} $\Xi$ of $K$ is a continuous homomorphism $\Xi:\ideles \to \CC^{\times}$ that is trivial on $K^{\times}$. In other words, a \emph{Hecke character} of $K$ is a continuous homomorphism $\Xi: K^{\times}\backslash\ideles \to  \CC^{\times}$.
        \item We say a Hecke character $\Xi$ is \emph{algebraic} if for each embedding $\kappa: K \hookrightarrow\CC$, there exists $n_{\kappa}\in \ZZ$ such that $\Xi(x)=\prod_{\kappa}(\kappa(x))^{-n_{\kappa}}$ for each $x$ in the connected component of the identity in $K_{\infty}^{\times}$. 
        \item Let $\Xi:K^{\times}\backslash\ideles \to  \CC^{\times}$ be an algebraic Hecke character of $K$. We say that $\Xi$ has \emph{infinity type} $(q,r)\in \ZZ^{2}$ if $\Xi_{\infty}(z)=z^{q}\overline{z}^{r}$, where for each place $v$ of $K$, we let $\Xi_{v}: K_{v}^{\times} \to \CC^{\times}$ be the $v$-component of $\Xi$.
\item The conductor of a Hecke character $\Xi$ is an ideal $\mathfrak{f}\coloneqq \prod_{\PP}\PP^{n_{\PP}}$, where the product runs over all primes of $K$, such that
 \begin{itemize}\item $n_{\PP}=0$ for almost all primes $\PP$,
\item for finitely many primes $\PP$, $n_{\PP}\in\ZZ_{>0}$and $\Xi_{\PP}(1+\PP^{n_{\PP}})=1$. Moreover, $n_{\PP}$ is minimal with this property. \end{itemize}
\end{enumerate}        
\end{defi}

From now on, all the Hecke characters mentioned in this article are algebraic Hecke characters. 
We can view Hecke characters as $p$-adic characters:
\begin{defi}[$p$-adic avatar of an algebraic Hecke character]
    Let $\Xi$ be a Hecke character of $K$ of conductor $\PP^{n_{\PP}}\PPP^{n_{\PPP}}$ and infinity type $(a,b)$.
    The $p$-adic avatar of $\Xi$ is defined as
        $$\widehat{\Xi}(x)\coloneqq x_{\PP}^{a}x_{\PPP}^{b}\cdot i_{p}i^{-1}_{\infty}(\Xi(x_{fin})),$$
where $x\in\mathbb{A}_{K}^{\times}$, and for $\qq\in\pset$, $x_\qq$ are the $\qq$ componenets of $x$. See also \cite[Section 7.3]{Will}.
\end{defi}


By the class field theory, the correspondence $\Xi \mapsto \widehat{\Xi}$ establishes a bijection between the set of algebraic Hecke characters of $K$ of conductor dividing $(p^\infty)$ and the set of locally algebraic $\overline{\Qp}$-valued characters of $G_{p^{\infty}}$. See \cite[Theorem 7.3]{Will}.

Now we combine this $p$-adic avatar of the Hecke character and the global Artin reciprocity map $\rec$ to define a Galois character on the ray class group $G_{p^{\infty}}$.

Let $\Xi$ be a Hecke character of $K$ of conductor $\PP^{n_{\PP}}\PPP^{n_{\PPP}}$ and infinity type $(a,b)$.
\begin{defi}
    The $p$-adic Galois character of $\Xi$ is $\widetilde{\Xi}: G^{ab}_{K} \to \Cp^{\times}$ given by 
        $$\widetilde{\Xi}= \widehat{\Xi}\circ\rec^{-1},$$
        where $\rec:\mathbb{A}^{\times}_{K}/K^\times \to G^{\mathrm{ab}}_{K}$ is the global Artin reciprocity. 
\end{defi}
Note that $G_{p^{\infty}}\subset G^{ab}_{K}$ and Remark \ref{rem5.1} implies  $\rec^{-1}(\gamma_{\qq})=\Art_{\qq}^{-1}(\gamma_{\qq})$.

By an abuse of notation, let $\tilX$ denote $\tilX\mid_{G_{p^{\infty}}}$.
Define $\tilX_{\qqq} \coloneqq \tilX|_{\overline{\langle \gamma_{\qqq} \rangle}}$.
Hence, if $\Xi$ is a Hecke character of the infinity type $(a,b)$ and conductor $\PP^{n_{\PP}}\PPP^{n_{\PPP}}$, then 
$$\tilX_{\PP}(\gamma_{\PP})=u^{a}\zeta_{\PP},$$ 
and 
$$\tilX_{\PPP}(\gamma_{\PPP})=u^{b}\zeta_{\PPP},$$ 
where $\zeta_{\qqq}$ is a $(p^{n_{\qqq}-1})$-th root of unity. 
In other words, for any $\sigma\in \Delta_{K},$
\begin{equation}
    \tilX(\sigma\times\gamma_{\PP}\times\gamma_{\PPP})=\tilX_{\PP}(\gamma_{\PP})\cdot\tilX_{\PPP}(\gamma_{\PPP})\times\text{some root of unity}.
\end{equation}
From the above discussion, we get
$$\widecheck{\Xi}\coloneqq\tilX |_{\overline{\langle \gamma_{\PP}\rangle}\times\overline{\langle \gamma_{\PPP}\rangle}}=\tilX_{\PP}\cdot\tilX_{\PPP}.$$

Moreover, for any Hecke character $\omega_{1}$ of infinity type $(a,0)$ and conductor $\PP^{n_{\PP}}$ and Hecke character $\omega_{2}$ of infinity type $(0,b),$ and conductor $\PPP^{n_{\PPP}}$, then the product $\widetilde{\omega_{1}}_{\PP}\cdot\widetilde{\omega_{2}}_{\PPP}$ is a character on $\overline{\langle \gamma_{\PP} \rangle} \times \overline{\langle \gamma_{\PPP} \rangle}$. 

\subsection{Two variable distribution modules}
\label{twovar}
We will extend $\HH_{E}(\GG_1)$ from previous sections. 
Let 
$$\HH_{\Cp,r}(\GG_1)=\{f(\gamma_{0}-1)\colon f(X)=\sum_{n\geq 0} a_{n}X^{n}\in \Cp[[X]], \sup_{n}(n^{-r}|a_{n}|_{p}) < \infty\}.$$
and
$$\HH_{\Cp}(\GG_1) = \bigcup_{r\geq 0}\HH_{\Cp,r}(\GG_1).$$
Note that $\HH_{E}(\GG_1)= E[[\gamma_{0}-1]] \bigcap \HH_{\Cp}(\GG_1)$, for any finite field extension $E$ of $\Qp$.

For $\qq\in\pset$, define a map $\tau_{\qqq}=\Art_{\qqq}\circ\chi$  between $\GG_1$ and $\GG_{\qqq}\coloneqq \Gal(H(\qqq^{\infty})/K)$ which sends $\gamma_{0}$ to $\gamma_{\qqq}$. 
This \emph{change of variable map} can be extended to ring isomorphisms
\begin{align*}
    \tau_{\qqq}\colon\HH_{\Cp}(\GG_1) &\to \HH_{\Cp}(\GG_{\qqq}),\\
                f(\gamma_{0}-1) &\mapsto f(\gamma_{\qqq}-1).
\end{align*}
Similarly, for any finite field extension $E$ of $\Qp$, we again have an isomorphism
$$\tau_{\qqq}: \HH_{E}(\GG_1) \to \HH_{E}(\GG_{\qqq}).$$

For any finite extension $E$ of $\Qp$, let $$\Lambda_{E}(G_{p^{\infty}}) = E[\Delta_{K}]\otimes_{\OO_E} \Lambda_{\Oe}(\GG_{\PP})\widehat{\otimes}_{\OO_E}\Lambda_{\Oe}(\GG_{\PPP}),$$
where
$\Lambda_{\Oe}(\GG_{\qqq})$ is the Iwasawa algebra $\OO_{E}[[\GG_{\qqq}]]$ for $\qqq\in\pset$. 
The two-variable Iwasawa algebra $\Lambda_{E}(G_{p^{\infty}})$ is isomorphic to the power series ring  
\[E[\Delta_{K}]\otimes_{\Oe} \Oe[[T_{1}, T_{2}]],\]
by identifying $\gamma_{\PP}-1$ with $T_{1}$ and $\gamma_{\PPP}-1$ with $T_{2}$.

We define the 
$$\HH_{\Cp,r,s}\coloneqq \HH_{\Cp,r}(\GG_{\PP})\widehat{\otimes}_{\CC_p}\HH_{\Cp,s}(\GG_{\PPP}),$$
and
$$\HH_{E,r,s}\coloneqq \HH_{\Cp,r,s}\cap E[[\gamma_{\PP}-1, \gamma_{\PPP}-1]],$$
for any finite extension $E$ of $\Qp$.
Thus
\[\HH_{E,r,s}= \HH_{E,r}(\GG_{\PP})\widehat{\otimes}_{E}\HH_{E,s}(\GG_{\PPP}).\]
We also define
$$\HH_{\Cp,r,s}(G_{p^{\infty}})\coloneqq \Cp[\Delta_{K}]\otimes_{\CC_p} \HH_{\Cp,r,s},$$
and
$$\HH_{E,r,s}(G_{p^{\infty}})\coloneqq E[\Delta_{K}]\otimes_{E} \HH_{E,r,s},$$
where $\Delta_K$ is the finite abelian group appearing in the Galois group $G_{p^{\infty}}$.

Note that
$$\Lambda_{E}(G_{p^{\infty}}) = \HH_{E,0,0}(G_{p^{\infty}}),$$
since we can identify $\HH_{E,0}(\GG_\qq)$ with $\Lambda_{E}(\GG_\qq)$ for $\qq\in\pset$.
Moreover, $\HH_{E,r,s}$ is a $\Lambda_{E}(\GG_{\PP})\widehat{\otimes}_{E}\Lambda_{E}(\GG_{\PPP})$-module, and $\HH_{E,r,s}(G_{p^{\infty}})$ is a $\Lambda_{E}(G_{p^\infty})$-module.


For $F \in \HH_{E,r,s}$, and a Hecke character $\Xi$ of the infinity type $(a,b)$ and conductor $\PP^{n_{\PP}}\PPP^{n_{\PPP}}$, define 
\begin{align}
    F^{(\tilX_{\PP})} &\coloneqq F(u^{a}\zeta_{\PP}-1, \gamma_{\PPP}-1),\label{ps1}\\
    F^{(\tilX_{\PPP})} &\coloneqq F(\gamma_{\PP}-1, u^{b}\zeta_{\PPP}-1),\label{ps2}
\end{align}
where $\zeta_{\qqq}$ is a primitive $p^{(n_{\qq}-1)}$-th root of unity for $\qqq\in\pset$.

\begin{lem}
\label{lem3.1}
    Let $E$ be a finite extension of $\Qp$, $F\in\HH_{E,r,s}$ be a power series and let $\Xi$ be a Hecke character of $K$ of infinity type $(a,b)$ and conductor $\PP^{n_{\PP}}.\PPP^{n_{\PPP}}$. 
    Then, $F^{(\tilX_{\PP})} \in \HH_{E(\tilX_{\PP}(\gamma_{\PP})),s}(\GG_{\PPP})$ and $F^{(\tilX_{\PPP})} \in \HH_{E(\tilX_{\PPP}(\gamma_{\PPP})),r}(\GG_{\PP})$, where $E(\tilX_{\qqq}(\gamma_{\qqq}))$ is finite extension of E by adjoining values of $\tilX_{\qqq}(\gamma_{\qqq})$ for $\qqq\in\pset$. 
    
\end{lem}

\begin{proof}
    The power series $F$ belongs to the power series ring $\HH_{E,r,s}$ which is completed tensor product of $\HH_{E,r}(\GG_{\PP})$ and $\HH_{E,s}(\GG_{\PPP})$.
    Thus, if we substitute $\gamma_{\PP}$ by $u^{a}\zeta_{\PP}$, then $F^{(\tilX_{\PP})}$ will be a power series with growth $s$ and the coefficients of $F^{(\tilX_{\PP})}$ will be in $E(\tilX_{\PP}(\gamma_{\PP}))$. 

    Similarly, substituting $\gamma_{\PPP}$ by $u^{b}\zeta_{\PPP}$, then $F^{(\tilX_{\PPP})}\in\HH_{E(\tilX_{\PPP}(\PPP)),s}(\GG_{\PP})$. 
\end{proof}

\textit{Isotypical components.} Let $\eta:\Delta_{K} \to \Zp^\times$ (or $\eta:\Delta_{K} \to \overline{\Qp}^\times$) be a character, where $\Delta_{K}$ is a finite abelian subgroup appearing in $G_{p^\infty}$. 
Write $e_{\eta}= \frac{1}{|\Delta_{K}|}\sum_{\sigma\in\Delta}\eta^{-1}(\sigma)\sigma$. 
For $*\in\{\Cp,E\}$, if $F\in \HH_{*,r,s}(G_{p^{\infty}})$, write $F^{\eta}=e_{\eta}(F)$ for its image in $e_{\eta}(\HH_{*,r,s}(G_{p^\infty}))\cong\HH_{*,r,s}$. 
Note that this isomorphism is a module isomorphism. 
If $\eta=1$, we simply write $F^{\Delta_{K}}$ instead of $F^{\eta}$. Note that $F^{\Delta_{K}}\in \HH_{*,r,s}$.

After this point, our integer $k$ will be greater than or equal to $0$ rather than $2$, since we will be working with Bianchi modular forms which are \emph{cohomological} in nature.

\section{Cuspidal Bianchi modular forms and their $p$-adic $L$-functions}
In this section, we will briefly recall the definition of Bianchi modular forms, their $L$-functions, and the $p$-adic $L$-functions constructed by Williams in \cite{Will}.

\subsection{Bianchi modular forms}
We define Bianchi modular forms adelically. Fix a quadratic imaginary field $K$. Let $\sigma$ denote an embedding $K\hookrightarrow\CC$, and let $c$ denote the complex conjugate, i.e., $c=\overline{\sigma}$.

\begin{defi}[Level]
Let $\widehat{\OO_{K}}\coloneqq \OO_{K}\otimes\widehat{\ZZ}$.
For any ideal $\frkm\in\OO_{K}$, define
\begin{enumerate}
\item   $ \Omega_{0}(\frkm) \coloneqq \left\{ \begin{pmatrix} a & b \\ c & d \end{pmatrix} \in \mathrm{GL}_{2}(\widehat{\mathcal{O}_{K}}) \colon c\in \frkm\widehat{\OO_K} \right\},$

\item $ \Omega_{1}(\frkm) \coloneqq \left\{ \begin{pmatrix} a & b \\ c & d \end{pmatrix} \in \mathrm{GL}_{2}(\widehat{\mathcal{O}_{K}}) \colon c\in \frkm\widehat{\OO_K}, a, d\in 1+\frkm\widehat{\OO_{K}} \right\}.$

\end{enumerate}
\end{defi}
Fix $k,\ell\in\ZZ_{\geq 0}$.
For any ring $R$, $V_{k+\ell+2}(R)$ denotes the space of homogeneous
polynomials over $R$ in two variables of degree $k+\ell+2$.

Let $\varepsilon$ be a Hecke character of infinity type $(-k,-\ell)$ and conductor dividing the level $\frkm$. 
For $x_{f}=\begin{pmatrix}
    a & b \\
    c& d
\end{pmatrix} \in \Omega_{0}(\frkm)$, set $\varepsilon_{\frkm}(x_{f})=\varepsilon_{\frkm}(d)=\prod_{\qq \mid \frkm}\varepsilon_{\qq}(d_{\qq})$.

\begin{defi}[Bianchi modular forms]
    We say a function $$\mathcal{F} : \mathrm{GL}_{2}(\mathbb{A}_{K}) \to V_{k+\ell+2}(\CC )$$ is a cuspform of weight $(k,\ell)$, level $\frkm$ (i.e. level $\Omega_{1}(\frkm)$), and central action of $\varepsilon$ if it satisfies
    \begin{enumerate}
        \item $\FF(zg)=\varepsilon(z)\FF(g)$ for $z\in\mathbb{A}^{\times}_{K}\cong Z(\mathrm{GL}_{2}(\mathbb{A}^{\times}_{K}))$;
        \item For all $u=u_{f}\cdot u_{\infty}\in \Omega_{0}(\frkm) \mathrm{SU}_{2}(\CC)$, $$\FF(gu)=\varepsilon_{\frkm}(u_{f})\FF(g)\rho_{k+\ell+2}(u_\infty),$$ 
        where $\rho: \mathrm{SU}_{2}(\CC) \to \mathrm{GL}(V_{k+\ell+2}(\CC))$ is an irreducible right representation;
        \item The function $\FF$ is left-invariant under the group $\mathrm{GL}_{2}(K)$.
        \item The function $\FF$ is an eigenfunction of Casimir operators $\partial_{\mathrm{id}}$ and $\partial_{c}$, with eigenvalues $k^{2}/2 + k$ and $\ell^{2}/2+\ell$ respectively. Here $\partial_{\mathrm{id}}/4, \partial_{c}/4$ are components of the Casimir operator in the Lie algebra $\mathfrak{sl}_{2}(\CC)\otimes_{\mathbb{R}}\CC$. 
       \item We say $\FF$ is a \emph{cuspidal} Bianchi modular form if $\FF$ satisfies the cuspidal condition for all $g\in \mathrm{GL}_{2}(\mathbb{A}_{K})$, that is, we have
        \begin{equation*}
            \int_{K\backslash\mathbb{A}_{K}} \FF(ug)\,du = 0,
        \end{equation*}
where we consider $\mathbb{A}_{K}$ inside $\gl(\mathbb{A}_{K})$, by the map sending $u\to\begin{pmatrix}1 & u \\ 0 & 1\end{pmatrix}$, and $\,du$ is the Lebesgue measure on $\mathbb{A}_K$.
    \end{enumerate}
\end{defi}
See \cite[Definition 1.2]{Will} and \cite[Definition 2.1]{LPS2} for the details about the definition of Bianchi modular forms.
\begin{rems}
\hfill
    \begin{enumerate}
        \item In a result by Harder, he showed that if $\FF$ is a non-zero cuspidal Bianchi modular form of weight $(k,\ell)$, then $k=\ell$. See \cite{Hard}.
        \item A Bianchi modular form $\FF$ descends to give a collection of $h$ automorphic forms $\FF^{i}: \mathbb{H}_{3} \to V_{k+\ell+2}(\CC)$, where $h$ is the class number of $K$, and $\mathbb{H}_3 \coloneqq  \RR_{\geq 0}\times \CC \cong \gl(\CC)/[\mathrm{SU}_{2}(\CC)Z(\mathrm{GL}_{2}(\CC))]$ is the hyperbolic $3$-space.
    \end{enumerate}
\end{rems}
From now onwards, we will deal with cuspidal Bianchi modular forms, and hence we will fix $k=\ell$ in $\ZZ_{\geq 0}$.
We will decompose $p$-adic $L$-functions associated to non-parallel weight Bianchi modular forms (of level $\Omega_{0}(\frkm)$) in Appendix \ref{appA}.

For a non-trivial cuspform $\FF$ of level $\frkm$ and weight $(k,k)$, we have the following Fourier expansion:
\begin{equation}
\label{fourier}
\FF\left[\begin{pmatrix}
    t & z \\ 0 & 1
\end{pmatrix}\right]= |t|_{K}\sum_{a\in K^{\times}} c(at\mathcal{D}_{K},\FF)W(at_{\infty})\mathrm{e}_{K}(az),
\end{equation}
where $\mathcal{D}_{K}$ is a generator of the different of $K$, $\mathrm{e}_{K}$ is an additive character of $K\backslash\mathbb{A}_K$, and $W$ is the Whittaker function. 
Note that the Fourier coefficients $c(\cdot, \FF)$ are functions on the fractional ideals of $K$, and $c(I,\FF)=0$ if $I$ is not an integral ideal. 
See \cite[Section 1.2]{Will}, \cite[Section 2.3]{LPS2} for the details.
For the Whittaker function, see \cite[Section 2.3]{LPS}.

\begin{defi}[Twisted $L$-function of a Bianchi modular form]
Let $\FF$ be a cuspidal Bianchi modular form of any weight and level $\frkm$. Let $\Xi$ be a Hecke character of conductor $\mathfrak{f}$. The twisted $L$-function of $\FF$ is defined as
\begin{equation*}
     L(\FF, \Xi, s) = \sum_{\substack{0\neq\mathfrak{a}\subset \OO_{K},\\(\mathfrak{f}, \mathfrak{a})=1}} c(\mathfrak{a},\FF)\Xi(\mathfrak{a})\mathrm{N}(\mathfrak{a})^{-s},
\end{equation*}
    where $c(\mathfrak{a},\FF)$ is the $\mathfrak{a}$-th Fourier coefficient of $\FF$ and $s\in\CC$ in some suitable right-half plane. 
\end{defi}
We also consider the completed $L$-function
$$\Lambda(\FF,\Xi, s)= \dfrac{\Gamma(q+s)\Gamma(r+s)}{(2\pi i)^{q+s}(2\pi i)^{r+s}} L(\FF,\Xi,s),$$
where $(q,r)$ is the infinity type of $\Xi$, and $\GG(-)$ are the Deligne's $\GG$-factors.

\subsection{$p$-adic $L$-function of  Bianchi modular forms}
Let $K/\QQ$ be a quadratic imaginary field with class number $h_K$, discriminant $-D$.
Let $p$ be an odd prime splitting in $K$ as $(p)=\PP\PPP$.

\begin{thm}[Williams, {{\cite[Theorem 7.4]{Will}}}]
\label{thmwilm}
    Let $\mathcal{F}$ be a cuspidal Bianchi eigenform over $K$ of weight $(k,k)$ and level $\frkn$ such that $(p)\mid\frkn$. 
Let $a_{\qq}$ denote the $U_{\qq}$-eigenvalues of $\mathcal{F}$ where $v_{p}(a_{\qq})< (k+1)$ for all $\qq\mid p$.
For any ideal $\mathfrak{f}$, we define the operator $U_{\mathfrak{f}}$ as
$$U_{\mathfrak{f}}\coloneqq \prod_{\PP^{n}\mid\mid \mathfrak{f}} U_{\PP}^{n}.$$
Then there exists a locally analytic distribution $L_{p,\mathcal{F}}$ on the ray class group $G_{p^\infty}$ such that for any Hecke character $\Xi$ of infinity type $(0,0)\leq (a,b)\leq (k,k)$ and conductor $\mathfrak{f}$, we have
\begin{equation}
    L_{p,\mathcal{F}}(\tilX)=(\text{explicit factor})\dfrac{1}{a_{\mathfrak{f}}}\dfrac{\Lambda(\mathcal{F},\Xi,1)}{\Omega_{\mathcal{F}}},
\end{equation}
where $a_{\mathfrak{f}}$ is $U_{\mathfrak{f}}$-eigenvalue of $\mathcal{F}$, and $\Omega_{\mathcal{F}}$ is a complex period.
Furthermore, the explicit factor is given by 
\[\prod_{\PP\mid p} e(\PP)\cdot \frac{\Xi(x_{\mathfrak{f}})Dw\tau(\Xi^{-1})}{(-1)^{a+b+2}2\Xi_{\mathfrak{f}}(x_{\mathfrak{f}})},\]
where $-D$ is the discriminant of $K$, $\tau(\cdot)$ is the Gauss sum, $w=|\OO_{K}^\times|$, and $e_\PP$ is described in \cite[Equation 3.2]{LPS}.

The distribution $L_{p,\FF}$ is $(v_{p}(a_{\PP}))_{\PP\mid p}$-admissible and is unique with these interpolation and growth properties.
\end{thm}

See \cite[Definition 6.14]{Will} for the definition of admissibility. 
Naively we say a distribution $\mu$ is $(h_1, h_2)$-admissible if $\mu$ is $O(\log_{p}^{h_{1}})$ in one variable (corresponding to $\PP$) and is $O(\log_{p}^{h_{2}})$ in the other variable (corresponding to $\PPP$).

\begin{rem}
Although we have stated Theorem \ref{thmwilm} for the $p$ split case, Williams proved it without any conditions on $p$ in \cite{Will} and for $p=2$ as well.
\end{rem}
\begin{rem}
In \cite[Section 6.3]{Will}, Williams has defined the admissibility for locally analytic distributions on $\OO_{K}\otimes\Zp$. 
But using methods from \cite{Loefller1}, we can extend the notion of admissibility for locally analytic distributions on $\OO_{K}\otimes\Zp$ to locally analytic distributions on ray class group $G_{p^{\infty}}$.
See \cite[Section 7.4]{Will} for more details.
\end{rem}

\begin{rem}
\label{rem7.7}
For simplicity, assume $\Delta_{F}$ to be trivial and $G_{p^\infty} \cong \GG_{\PP}\times\GG_{\PPP}\cong \Zp\times\Zp$.
For real numbers $r,s\geq 0$ and let $E$ be a finite extension of $\Qp$, we define $D^{(r,s)}(G_{p^{\infty}},E)$ to be the set of distributions $\mu$ of $G_{p^\infty}$ such that for $m,n\in\ZZ_{\geq 0}$ and integers $0\leq i \leq \lfloor r \rfloor$ and $0\leq j\leq \lfloor s \rfloor$ ,
$$ v_{p}\left(\int_{\substack{a+p^{m}\Zp \\ b+p^{n}\Zp}} (x-a)^{i}(y-b)^{j} d\mu\right) \geq R-(r-i)m-(s-j)n,$$
for some constant $R\in\RR$ which only depends on $\mu$. 
See \cite[Theorem 2.3.2]{Colmez} for the analogous definitions in the one-variable setting. 
    
Since $p$ splits in the quadratic imaginary field $K$, we can identify $D^{(u,v)}(G_{p^\infty},E)$ with $\HH_{E,u,v}(G_{p^\infty})$ via Amice transform $\mu\mapsto A_{\mu}(T_{1}, T_{2})=\displaystyle\int_{\Zp\times\Zp} (1+T_{1})^{x}(1+T_{2})^{y}\mu(x,y)$. 
\end{rem}

Therefore if $p$ splits in $K$ as $(p)=\PP\PPP$, and $\FF$ is a Bianchi modular form which satisfies the conditions of the above theorem, then $L_{p,\mathcal{F}}$ is a $(v_{p}(a_{\PP}),v_{p}(a_{\PPP}))$-admissible locally analytic distribution on $G_{p^{\infty}}$, that is, by \cite{Will} and \cite{Loefller1}, 
\begin{equation}
    L_{p,\mathcal{F}}\in D^{(v_{p}(a_{\PP}),v_{p}(a_{\PPP}))}(G_{p^\infty},F),\label{admeqn}
\end{equation}
where $F$ is a finite extension of $\Qp$ containing $a_{\PP}$ and $a_{\PPP}$.
By Remark \ref{rem7.7}, we can view
\[ L_{p,\mathcal{F}}\in \HH_{F, v_{p}(a_{\PP}),v_{p}(a_{\PPP})}(G_{p^\infty}). \]

\section{Factorisation of $p$-adic $L$-functions of Bianchi modular forms}
In this section, we first define $p$-stabilizations of cuspidal Bianchi modular forms. 
Next, we modify our logarithmic matrix $\Mbar$ and prove the main factorization theorem of two-variable $p$-adic $L$-functions. 
From now we fix the imaginary quadratic field $K$ and all the cuspidal Bianchi modular forms $\FF$ we consider are over $K$. Also, the fixed odd prime $p$ splits in $K$ as $\PP\PPP$.
\subsection{$p$-stabilization of Bianchi modular forms}
We begin with a cuspidal Bianchi modular form $\FF$ of weight $(k,k)$ and level $\frkm$ such that $(p\OO_K, \frkm)=1$. 
We further assume $\FF$ is a Hecke eigen cuspform and for $\qq\in \{\PP,\PP\}$, we have $T_{\qq}\FF=a_{\qq}\FF$. 
Note that the norm of $\qq$ is $p$.
Consider the Hecke polynomial $X^{2}-a_{\qq}X+ u_{\qq}\cdot p^{k-1} = (X-\alpha_{\qq})(X-\beta_{\qq})$, where $u_{\qq}$ is some $p$-adic unit coming from the nebentypus of $\FF$.

There are four $p$-stabilisations $\mathcal{F}^{\alpha_{\PP},\alpha_{\PPP}}, \mathcal{F}^{\alpha_{\PP},\beta_{\PPP}},\mathcal{F}^{\beta_{\PP},\alpha_{\PPP}},$and $ \mathcal{F}^{\beta_{\PP},\beta_{\PPP}}$ of level $p\mathfrak{m}$, such that for $*\in \{\alpha_{\PP},\beta_{\PP}\}$ and $\dagger \in \{\alpha_{\PPP},\beta_{\PPP}\}$, we have
\begin{align}
    U_{\PP}(\FF^{*,\dagger})&=*\FF^{*,\dagger},\\
    U_{\PPP}(\FF^{*,\dagger})&=\dagger\FF^{*,\dagger}.
\end{align}

For more details about $p$-stabilizations, refer to \cite[Section 3.3]{LPS}.

\begin{assum}
    Throughout the article, we assume 
    \begin{enumerate}
       \item $\FF$ is good non-ordinary at both the primes $\PP$ and $\PPP$ i.e. $v_{p}(a_{\PP})$ and $v_{p}(a_{\PPP}) > 0$.
        \item $v_{p}(a_{\qq}) > \left\lfloor \dfrac{k}{p-1} \right\rfloor$ and $\alpha_{\qq}\neq\beta_{\qq},$ for $\qq\in\pset$.
    \end{enumerate}    
\end{assum}
Note that for any $x\in\{\alpha_{\PP}, \beta_{\PP}, \alpha_{\PPP}, \beta_{\PPP}\},$ we know $v_{p}(x)< k+1$.

Let $E$ be a finite extension of $\Qp$ which contains $\alpha_{\PP},\alpha_{\PPP},\beta_{\PP}$ and $\beta_{\PPP}$, and $K$.
\subsection{Modified logarithmic matrices}
\label{modi}
Recall the ring isomorphism $\tau_{\qqq}:\HH_{F}(\GG_1) \to \HH_{F}(\GG_{\qqq})$, for $\qq\in\pset$.
Consider the change of variable map between matrices induced by $\tau_{\qqq}$:
$$\mat_{\qqq}:M_{2,2}(\HH_{F}(\GG_1)) \to M_{2,2}(\HH_{F}(\GG_{\qqq})),$$
such that 
$$\mat_{\qqq}\left( \begin{pmatrix} A & B \\[6pt] C & D \end{pmatrix}\right) = \begin{pmatrix} \tau_{\qqq}(A) & \tau_{\qqq}(B) \\[6pt] \tau_{\qqq}(C) & \tau_{\qqq}(D) \end{pmatrix}.$$
Note that this map is also a ring isomorphism.

Let $\mathcal{F}$ be a Bianchi modular form of level $\mathfrak{m}$ coprime to $p$ and let $\alpha_{\qqq}$ and $\beta_{\qqq}$ be the $T_{\qqq}$-eigenvalues, for $\qqq\in\pset$.
From the Sections \ref{sec3}, \ref{sec4}, and \ref{sec5}, we construct a logarithmic matrix $\underline{M(\qq)}\in M_{2,2}(\HH_{E}(\GG_1))$ using $\alpha_{\qq}$ and $\beta_{\qq}$, since $v_{p}(a_{\qq})=v_{p}(\alpha_{\qq}+\beta_{\qq})> \left\lfloor \dfrac{k}{p-2}\right\rfloor$ and $\alpha_{\qq}\beta_{\qq}=u_{\qq} p^{k+1}$. Let $E/\Qp$ be a finite extension large enough to contain $\alpha_{\qq}, \beta_{\qq}$ and a fixed square root of $u_{\qq}$.
Let 
\begin{align*}
    A_{\vp,\qqq}&=\begin{pmatrix}
        0 & \dfrac{-1}{u_{\qq}p^{k+1}} \\[12pt]  1 & \dfrac{a_{\qqq}}{u_{\qq}p^{k+1}},
    \end{pmatrix},\\[10pt]
    Q_{\qqq }&= \begin{pmatrix} \alpha_{\qqq} & -\beta_{\qqq} \\[6pt]
-p^{k+1}u_{\qq} & p^{k+1}u_{\qq} \end{pmatrix},\\[10pt]
\underline{M_{\qqq}} &= \mat_{\qqq}(\underline{M(\qq)}).
\end{align*}
\begin{rem}
    From Proposition \ref{prop4} and Lemmas \ref{lem5}, \ref{lem1.16}, we deduce:
    \begin{enumerate}
        \item If $Q_{\qq}^{-1}(\underline{M(\qq)})=\begin{pmatrix}
            P_{1}(\qq) & P_{2}(\qq) \\
            P_{3}(\qq) & P_{4}(\qq)
        \end{pmatrix},$
        then $P_{1}(\qq), P_{2}(\qq)\in \HH_{E,v_{p}(\alpha_{\qq})}(\GG_1)$ and $P_{3}(\qq), P_{4}(\qq)\in \HH_{E,v_{p}(\beta_{\qq})}(\GG_1)$.
        \item The second row of $A_{\varphi, \qq}^{-n}\underline{M(\qq)}$ is divisible by $\Phi_{n-1,k+1}(\gamma_{0})$ over $\HH_{E}(\GG_1)$.
        \item The determinant $\det(\underline{M(\qq)})$ is $\dfrac{\log_{p,k+1}(\gamma_{0})}{\delta_{k+1}(\gamma_{0}-1)}$, up to a unit in $\Lambda_{E}(\GG_1)$. 
    \end{enumerate}
\end{rem}
\begin{rem}
    Here we get $k+1$ since $k\in\ZZ_{\geq 0}$ and not $k\in\ZZ_{\geq 2}$.
\end{rem}

\begin{thm}
\label{thm3.2}
For $\qq\in\pset$ the following are true,
    \begin{enumerate}
        \item The elements in the first row of $Q_{\qqq}^{-1}\underline{M_{\qqq}}$ are inside $\HH_{E, v_{p}(\alpha_{\qqq})}(\GG_{\qqq})$, while the elements in the second row  are in $\HH_{E, v_{p}(\beta_{\qqq})}(\GG_{\qqq})$.\label{3.2.1}
        \item The second row of the matrix $A_{\vp,\qq}^{-n}\underline{M_{\qqq}}$ is divisible by the cyclotomic polynomial $\Phi_{n-1,k+1}(\gamma_{\qqq})$.\label{3.2.2}
        \item The determinant of $\underline{M_{\qqq}}$ is $\dfrac{\log_{p,k+1}(\gamma_{\qqq})}{\delta_{k+1}(\gamma_\qqq-1)}$ upto a unit in $\Lambda_{E}(\GG_{\qqq})$.\label{3.2.3}
    \end{enumerate}
\end{thm}

\begin{proof}
    First statement follows from Proposition \ref{prop4} and the definitions of $Q_{\qqq}$ and $\underline{M_{\qqq}}$.
    Second statement follows from Lemma \ref{lem5}.

    For the last statement, from the definitions we have $\underline{M_{\qqq}}=\mat_{\qqq}(\underline{M(\qq)})$. Thus, 
    \begin{align*}
        \det(\underline{M_{\qqq}})&=\det(\mat_{\qqq}(\underline{M(\qq)})),\\
                                 &= \tau_{\qqq}(\det(\underline{M(\qq)})),
    \end{align*}
    since $\tau_{\qqq}$ is a ring isomorphism.
    Hence the result follows from Lemma \ref{lem1.16}.
\end{proof}

For the rest of the article, we will write
\begin{equation}
    Q_{\qqq}^{-1}\underline{M_{\qqq}} = \begin{pmatrix} P_{1,\qqq} & P_{2,\qqq} \\[6pt] P_{3,\qqq} & P_{4,\qqq}\end{pmatrix}. \label{qm}
\end{equation}

\subsection{The main theorem and its proof}
In this section, we generalize the results of \cite{Lei3} and use them to decompose the two variable $p$-adic $L$-functions of Bianchi modular forms.

For the Bianchi cuspform $\mathcal{F}$ of level $\mathfrak{m}$ which is not divisible by $p$, let $a_{\qqq}$ be the $T_{\qqq}$-eigenvalue of $\mathcal{F}$ for $\qqq\in\pset$.
Recall that we have four $p$-stabilizations $\mathcal{F}^{\alpha_{\PP},\alpha_{\PPP}}, \mathcal{F}^{\alpha_{\PP},\beta_{\PPP}}, \mathcal{F}^{\beta_{\PP},\alpha_{\PPP}}$, and $ \mathcal{F}^{\beta_{\PP},\beta_{\PPP}}$ of level $p\frkm$, where $\alpha_{\qqq},\beta_{\qqq}$ are the roots of Hecke polynomial $X^{2}-a_{\qqq}X+u_{\qq}p^{k-1}$, for $\qqq\in\pset$.

Therefore, from Theorem \ref{thmwilm} and equation \eqref{admeqn}, for $*\in\{\alpha_{\PP},\beta_{\PP}\}$ and $\dagger\in\{\alpha_{\PPP},\beta_{\PPP}\}$, we have
$$L_{*,\dagger}\coloneqq L_{p,\mathcal{F}^{*,\dagger}} \in \HH_{E, v_{p}(*),v_{p}(\dagger)}(G_{p^\infty}),$$
since $v_{p}(*), v_{p}(\dagger)< k+1$.
Moreover, for any Hecke character $\Xi$ of infinity type $(0,0)\leq (a,b) \leq (k,k)$ and conductor $\PP^{n_{\PP}}\PPP^{n_{\PPP}}$ with $n_{\PP},n_{\PPP}\geq 1$, we have
\begin{align}
    L_{\alpha_{\PP},\alpha_{\PPP}}(\tilX)&=\alpha_{\PP}^{-n_{\PP}}\alpha_{\PPP}^{-n_{\PPP}}C_{a,b,\tilX}, \label{44}\\ 
    L_{\alpha_{\PP},\beta_{\PPP}}(\tilX)&=\alpha_{\PP}^{-n_{\PP}}\beta_{\PPP}^{-n_{\PPP}}C_{a,b,\tilX}, \label{45} \\
    L_{\beta_{\PP},\alpha_{\PPP}}(\tilX)&=\beta_{\PP}^{-n_{\PP}}\alpha_{\PPP}^{-n_{\PPP}}C_{a,b,\tilX}, \label{46} \\
    L_{\beta_{\PP},\beta_{\PPP}}(\tilX)&=\beta_{\PP}^{-n_{\PP}}\beta_{\PPP}^{-n_{\PPP}}C_{a,b,\tilX} \label{47},
\end{align}
where $C_{a,b,\tilX}\in \overline{\Qp}$ is a constant independent of $\alpha_{\PP}, \beta_{\PP}, \alpha_{\PPP}, \beta_{\PPP}$.
More precisely, $C_{a,b,\tilX}$ is $\text{(explicit factor)}\dfrac{\Lambda(\FF,\Xi,1)}{\Omega_\FF}$, since the conductor of $\Xi$ is $\PP^{n_{\PP}}\PPP^{n_{\PPP}}$ with $n_{\PP}, n_{\PPP}\in\ZZ_{\geq 0}$, we know $L(\FF, \Xi, 1)=L(\FF^{*,\dagger},\Xi,1)$ for all $p$-stabilizations $\FF^{*,\dagger}$ of $\FF$ and hence $C_{a,b,\tilX}$ is independent of $\alpha_{\PP},\beta_{\PP},\alpha_{\PPP},\beta_{\PPP}$. 

The main theorem is as follows:
\begin{thm}
    \label{mainthm2}
    There exist two variable power series with bounded coefficients, that is, there exist $L_{\sharp,\sharp}, L_{\sharp,\flat}, L_{\flat,\sharp}, L_{\flat,\flat}\in \Lambda_{E}(G_{p^{\infty}})$ such that
    \begin{equation}
        \begin{pmatrix}
            L_{\alpha_{\PP},\alpha_{\PPP}} & L_{\beta_{\PP},\alpha_{\PPP}}\\[6pt]
            L_{\alpha_{\PP},\beta_{\PPP}} & L_{\beta_{\PP},\beta_{\PPP}}
        \end{pmatrix}
        = Q^{-1}_{\PPP}\underline{M_{\PPP}} \begin{pmatrix}
            L_{\sharp,\sharp} & L_{\flat,\sharp}\\[6pt]
            L_{\sharp,\flat} & L_{\flat,\flat}
        \end{pmatrix} (Q^{-1}_{\PP}\underline{M_{\PP}})^{T}.
    \end{equation}
\end{thm}

\begin{rem}
    Theorem \ref{mainthm2} is analogous to \cite[Theorem 2.2]{Lei3} and \cite[Equation (24)]{BL1}.
\end{rem}

We first factorize through the variable $\gamma_{\PP}$ and then through $\gamma_{\PPP}$. 
In other words, we will first decompose the matrix $ \begin{pmatrix}
            L_{\alpha_{\PP},\alpha_{\PPP}} & L_{\beta_{\PP},\alpha_{\PPP}}\\[6pt]
            L_{\alpha_{\PP},\beta_{\PPP}} & L_{\beta_{\PP},\beta_{\PPP}}
        \end{pmatrix}$ in terms of matrix, say $C$, and $\Mbarp$. Then we decompose $C$ as a product of matrices $\begin{pmatrix}
            L_{\sharp,\sharp} & L_{\flat,\sharp}\\[6pt]
            L_{\sharp,\flat} & L_{\flat,\flat}
        \end{pmatrix}$ and $\Mbarpp$.  

First, we need the following classical result: 
\begin{lem}
\label{lem3.5}
    Let $\gamma$ be a topological generator of $\Zp$.
    Let $s,h$ be non-negative integers and assume $s<h$.
    If $F\in E[[\gamma-1]]$ is $O(\log_{p}^{s})$ and vanishes at all characters of type $\chi^{i}\omega$ for all $0\leq i \leq h-1$, where $\chi$ is any character which sends $\gamma$ to another topological generator $u\in 1 +p\Zp$ (for example the cyclotomic character) and $\omega$ is any Dirichlet character of conductor $p^n, n\geq1$,  then $F$ is identically $0$. 
\end{lem}
\begin{proof}
    See \cite[Lemme II.2.5]{AV} and \cite[Lemma 2.10]{Vishik}.

    We give a sketch here.
    We say a power series $G_1\in\HH_{E}$ is $o(G_2)$ if 
    \[\lv G_1 \lv_{\rho} = o(\lv G_2\lv_{\rho})\]
    as $\rho\to 1^{-}$.
    The power series $F$ is $o(\log_{p}^h)$, since $F$ is $O(\log_{p}^{s})$ and $s<h$.
    Suppose $F$ is not $0$ and, from the assumption, the zeros of $F$ are of the form $u^{j}\zeta-1$, where $\zeta$ is a $p^{n}$-th root of unity, for all $0 \leq j \leq h-1 $ and $n\geq1$. Then, by Remark \ref{logproplem}, we have 
    $$F=\log_{p,h}(\gamma)\cdot G,$$
    where $G\neq 0$ is another power series.
    Note that $\log_{p,h}(\gamma)$ is not $o(\log_{p}^h)$, and therefore $F$ is not $o(\log_{p}^h)$.
    This is a contradiction.
\end{proof}

Recall that $E$ is a finite extension of $\Qp$ containing $\alpha_{\PP}, \beta_{\PP}, \alpha_{\PPP}$ and $\beta_{\PPP}$. 
Let $S_{\PP}$ be the set of all Hecke characters on the ray class group $G_{p^{\infty}}$ with infinity type $(0,0)\leq(a,0)\leq(k,0)$ and conductor $\PP^{n_{\PP}}$, for $n_{\PP}>1$.
Similarly, let $S_{\PPP}$ be the set of all Hecke characters on the ray class group $G_{p^{\infty}}$ with infinity type $(0,0)\leq(0,b)\leq(0,k)$ and conductor $\PPP^{n_{\PPP}}$, for $n_{\PPP}>1$.

\begin{prop}
    \label{prop3.6}
    There exist $L_{\sharp, \alpha_{\PPP}},L_{\flat, \alpha_{\PPP}} \in \HH_{E, 0,v_{p}(\alpha_{\PPP})}(G_{p^{\infty}})$ and $L_{\sharp, \beta_{\PPP}},L_{\flat, \beta_{\PPP}} \in \HH_{E,0,v_{p}(\beta_{\PPP})}(G_{p^{\infty}})$ such that
    \begin{equation*}
        \begin{pmatrix}
            L_{\alpha_{\PP},\alpha_{\PPP}} & L_{\beta_{\PP},\alpha_{\PPP}}\\[6pt]
            L_{\alpha_{\PP},\beta_{\PPP}} & L_{\beta_{\PP},\beta_{\PPP}}
        \end{pmatrix} = \begin{pmatrix}
            L_{\sharp,\alpha_{\PPP}} & L_{\flat,\alpha_{\PPP}}\\[6pt]
            L_{\sharp,\beta_{\PPP}} & L_{\flat,\beta_{\PPP}}
        \end{pmatrix} (Q^{-1}_{\PP}\Mbarp)^T. 
    \end{equation*}
\end{prop}

\begin{proof}
This is a generalization of \cite[Proposition 2.3]{Lei3}.
Recall that, for $*\in\{\alpha_{\PP}, \beta_{\PP}\}$ and $\dagger\in\{\alpha_{\PPP},\beta_{\PPP}\}$, $L_{*,\dagger}$ are locally analytic distributions on the ray class group $G_{p^{\infty}}\cong \Delta_{K}\times \overline{\langle \gamma_{\PP} \rangle}\times \overline{\langle \gamma_{\PPP} \rangle}$.
For any character $\eta:\Delta_{K} \to \overline{\Zp}^{\times}$, we will prove that for $\Laa^{\eta}, \Lba^{\eta} \in \HH_{E,v_{p}(\alpha_{\PP}),v_{p}(\alpha_{\PPP})},$ there exist  $\Lsa^{\eta},\Lfa^{\eta}\in \HH_{E, 0,v_{p}(\alpha_{\PPP})}$ such that
    \begin{equation}
        \begin{pmatrix}
            \Laa^{\eta} \\[6pt]
            \Lba^{\eta}
        \end{pmatrix}
        =\matqp\Mbarp\begin{pmatrix}
            \Lsa^{\eta} \\[6pt]
            \Lfa^{\eta}
        \end{pmatrix}. \label{50}
    \end{equation}
and therefore,
\begin{equation*}
    \begin{pmatrix}
            \Laa \\[6pt]
            \Lba
        \end{pmatrix}
        =\matqp\Mbarp\begin{pmatrix}
            \Lsa \\[6pt]
            \Lfa
        \end{pmatrix}.    
\end{equation*}
Fix $\eta=1$.
The proof for other characters is similar.

Let $\Xi$ be any Hecke character of  $K$ of infinity type $(0,0) \leq (a,b) \leq (k,k)$ and conductor $\PP^{n_{\PP}}\PPP^{n_{\PPP}}$, where $n_{\PP}, n_{\PPP}\geq 1$.
Using equations \eqref{44} and \eqref{46}, we get
\begin{align}
    \Laa^{\Delta_{K}}(\widecheck{\Xi}) &= \alpha_{\PP}^{-n_{\PP}}(\alpha_{\PPP}^{-n_{\PPP}}\cdot C_{a,b,\tilX}),\label{51}\\[6pt]
    \Lba^{\Delta_{K}}(\widecheck{\Xi}) &= \beta_{\PP}^{-n_{\PP}}(\alpha_{\PPP}^{-n_{\PPP}}\cdot C_{a,b,\tilX}), \label{52}
\end{align}
where $\widecheck{\Xi}=\tilX|_{\overline{\langle \gamma_{\PP} \rangle}\times \overline{\langle \gamma_{\PPP} \rangle}} = \tilX_{\PP}\cdot\tilX_{\PPP}$.
Denote $\alpha^{-n_{\PPP}}C_{a,b,\tilX}$ by $D$.
Using \eqref{ps2}, we can rewrite equations \eqref{51} and \eqref{52} as
\begin{align}
    \Laa^{\Delta_{K}(\tilX_{\PPP})}(\tilX_{\PP}) &= \alpha_{\PP}^{-n_{\PP}} D, \label{53}\\[6pt]
    \Lba^{\Delta_{K}(\tilX_{\PPP})}(\tilX_{\PP}) &= \beta_{\PP}^{-n_{\PP}}D. \label{54}
\end{align}
 Note that $\Laa^{{\Delta_{K}(\tilX_{\PPP})}} \in \HH_{E', v_{p}(\alpha_{\PP})}(\GG_{\PP})$ and $\Lba^{{\Delta_{K}(\tilX_{\PPP})}} \in \HH_{E', v_{p}(\beta_{\PP})}(\GG_{\PP})$, where $E'=E(\tilX_{\PPP}(\gamma_{\PPP}))$ is a finite field extension of $E$.

Let $$G_{1} = P_{4,\PP}\Laa^{\Delta_{K}(\tilX_{\PPP})} - P_{2,\PP}\Lba^{\Delta_{K}(\tilX_{\PPP})},$$ and 
$$G_{2} = -P_{3,\PP}\Laa^{\Delta_{K}(\tilX_{\PPP})} + P_{1,\PP}\Lba^{\Delta_{K}(\tilX_{\PPP})},$$
where $Q^{-1}_{\PP}\Mbarp = \begin{pmatrix}
    P_{1,\PP} & P_{2,\PP}\\[1em] P_{3,\PP} & P_{4,\PP}
\end{pmatrix}\in\ M_{2,2}(\HH_{E}(\GG_{\PP})).$

Then, from Theorem \ref{thm3.2}, we get $G_{1},G_{2} \in \HH_{E', k-1}(\GG_{\PP})$, and equations \eqref{53} and \eqref{54} imply $G_{1}(\tilX_{\PP})=G_{2}(\tilX_{\PP})=0$. 
Hence, from Theorem \ref{thm1}, we deduce that $\det(Q_{\PP}^{-1}\Mbarp)$ divides both $G_{1}$ and $G_{2}$ in $\HH_{E',k-1}(\GG_{\PP})$.

Thus,
\begin{equation}
     P_{4,\PP}(\widecheck{\Xi})\Laa^{\Delta_{K}}(\widecheck{\Xi}) - P_{2,\PP}(\widecheck{\Xi})\Laa^{\Delta_{K}}(\widecheck{\Xi})= -P_{3,\PP}(\widecheck{\Xi})\Laa^{\Delta_{K}}(\widecheck{\Xi}) + P_{1,\PP}(\widecheck{\Xi})\Lba^{\Delta_{K}}(\widecheck{\Xi}) = 0, \label{58}
\end{equation}
for any Hecke character $\Xi$ of infinity type $(0,0) \leq (a,b) \leq (k,k)$ and conductor $\PP^{n_{\PP}}\PPP^{n_{\PPP}}$.

Thus, for any Hecke characters $\omega_{1}\in S_{\PP}$, $\omega_{2}\in S_{\PPP}$ , \eqref{58} implies
\begin{equation*}
    (P_{4,\PP}\Laa^{\Delta_{K}} - P_{2,\PP}\Lba^{\Delta_{K}})(\widetilde{\omega_{1}}_{\PP}\widetilde{\omega_{2}}_{\PPP})=0,
\end{equation*}
which we rewrite as
\begin{equation}
    (P_{4,\PP}\Laa^{\Delta_{K}} - P_{2,\PP}\Lba^{\Delta_{K}})^{(\widetilde{\omega_{1}}_{\PP})}(\widetilde{\omega_{2}}_{\PPP})=0.
\end{equation}
Similarly.
\begin{equation}
    (-P_{3,\PP}\Laa^{\Delta_{K}} + P_{1,\PP}\Lba^{\Delta_{K}})^{(\widetilde{\omega_{1}}_{\PP})}(\widetilde{\omega_{2}}_{\PPP})=0.
\end{equation}
Hence, the distributions $(P_{4,\PP}\Laa^{\Delta_{K}} - P_{2,\PP}\Lba^{\Delta_{K}})^{(\widetilde{\omega_{1}}_{\PP})}$ and   $(-P_{3,\PP}\Laa^{\Delta_{K}} + P_{1,\PP}\Lba^{\Delta_{K}})^{(\widetilde{\omega_{1}}_{\PP})}$ vanish at all characters $\omega_{2}\in S_{\PPP}$. 
Moreover, these two distributions belong to $\HH_{E', v_{p}(\alpha_{\PPP})}(\GG_{\PPP})$ and $v_{p}(\alpha_{\PPP}) < k+1$. 
Therefore, using Lemma \ref{lem3.5}, we get
\begin{align*}
    (P_{4,\PP}\Laa^{\Delta_{K}} - P_{2,\PP}\Lba^{\Delta_{K}})^{(\widetilde{\omega_{1}}_{\PP})}&=0,\\[6pt]
    (-P_{3,\PP}\Laa^{\Delta_{K}} + P_{1,\PP}\Lba^{\Delta_{K}})^{(\widetilde{\omega_{1}}_{\PP})} &=0.
\end{align*}

Hence, from the proof of Theorem \ref{thm1}, we conclude that $\det(Q_{\PP}^{-1}\Mbarp)$ divide $(P_{4,\PP}\Laa^{\Delta_{K}} - P_{2,\PP}\Lba^{\Delta_{K}})$ and $(-P_{3,\PP}\Laa^{\Delta_{K}} + P_{1,\PP}\Lba^{\Delta_{K}})$ over the two-variable distribution algebra $\HH_{E,k-1, v_{p}(\alpha_{\PPP})}$. 
More precisely, $\det(Q_{\PP}^{-1}\Mbar)$ divides both distributions in the variable $\gamma_{\PP}-1$ while not touching $\gamma_{\PPP}-1$.

Write 
\begin{align*}
    \Lsa^{\Delta_{K}} &= \dfrac{P_{4,\PP}\Laa^{\Delta_{K}} - P_{2,\PP}\Lba^{\Delta_{K}}}{\det(Q_{\PP}^{-1}\Mbarp)}, \\[6pt]
    \Lfa^{\Delta_{K}} &= \dfrac{-P_{3,\PP}\Laa^{\Delta_{K}} + P_{1,\PP}\Lba^{\Delta_{K}}}{\det(Q_{\PP}^{-1}\Mbarp)}.
\end{align*}
Then, since $\det(Q^{-1}_{\PP}\Mbarp) \sim \log^{k-1}_{p}$, Theorem \ref{thm1} implies $\Lsa^{\Delta_{K}}$ and $\Lfa^{\Delta_{K}}$ are $O(1)$ in $\gamma_{\PP}-1$ and hence lie in $\HH_{E,0,v_{p}(\alpha_{\PPP})}$.

We then write
\begin{align*}
    \Lsa &= \dfrac{P_{4,\PP}\Laa - P_{2,\PP}\Lba}{\det(Q_{\PP}^{-1}\Mbarp)}, \\[6pt]
    \Lfa &= \dfrac{-P_{3,\PP}\Laa + P_{1,\PP}\Lba}{\det(Q_{\PP}^{-1}\Mbarp)}.
\end{align*}
Since $\det(Q_{\PP}^{-1}\Mbarp)$ divides each isotypic component of the two distributions in the numerators, $\Lsa$ and $\Lfa$ are elements in $\HH_{E,0,v_{p}(\alpha_{\PPP})}(G_{p^{\infty}})$.

The proof for 
\[
        \begin{pmatrix}
            \Lab \\[6pt]
            \Lbb
        \end{pmatrix}
        =Q^{-1}_{\PP}\Mbarp\begin{pmatrix}
            \Lsb \\[6pt]
            \Lfb
        \end{pmatrix}.
    \]
    is identical. 
\end{proof}

Recall $\HH_{E,r}\coloneqq \{ f(X)=\sum_{n\geq 0} a_{n}X^{n} \in E[[X]] \colon \sup_{n}(n^{-r}|a_{n}|_{p}) < \infty\}$. 
Then, we can identify $\HH_{E,r,s} \coloneqq \HH_{E,r}(\GG_{\PP})\widehat{\otimes}\HH_{E,s}(\GG_{\PPP})$ with $\HH_{E,r}\widehat{\otimes}\HH_{E,s}$ by identifying $X=\gamma_{\PP}-1$ and $Y=\gamma_{\PPP}-1$. 
We define the operator $\delp$ to be the partial derivative $\delx$.
The next proposition is a generalization of \cite[Lemma 2.4]{Lei3}. 

\begin{prop}
    \label{prop3.7}
    Let $\Xi$ be any character Hecke character of $K$ of the infinity type $(0,0)\leq (a,b)\leq (k,k)$ and conductor $\PP^{n_{\PP}}\PPP^{n_{\PPP}}$ with $n_{\PP}, n_{\PPP} > 0$. 
    Then, there exist constants $A_{a,b,\Xi}, B_{a,b,\Xi}\in \overline{\Qp}$ such that 
    $$\begin{matrix}
        \delp\Laa(\tilX) =\alpha_{\PPP}^{-n_{\PPP}}A_{a,b,\tilX}, & \delp\Lab(\tilX) = \beta_{\PPP}^{-n_{\PPP}}A_{a,b,\tilX}, \\[6pt]
        \delp\Lba(\tilX) = \alpha_{\PPP}^{-n_{\PPP}}B_{a,b,\tilX}, & \delp\Lbb(\tilX) = \beta_{\PPP}^{-n_{\PPP}}B_{a,b,\tilX}. 
    \end{matrix}$$
\end{prop}

\begin{proof}
    We will only show that
    \begin{align*}
        \alpha_{\PPP}^{n_{\PPP}}\delp\Laa(\tilX)&=\beta_{\PPP}^{n_{\PPP}}\delp\Lab(\tilX),
    \end{align*}
    for any Hecke character $\Xi$ of  of the infinity type $(0,0)\leq (a,b)\leq (k,k)$ and conductor $\PP^{n_{\PP}}\PPP^{n_{\PPP}}$ with $n_{\PP}, n_{\PPP} > 0$. 

    Fix a Hecke character $\omega_{1} \in S_{\PP}$.
    Then, for any Hecke character $\omega_{2} \in S_{\PPP}$, \eqref{44} and \eqref{45} imply
    \begin{equation}
        \alpha_{\PPP}^{n_{\PPP}}\Laa^{\Delta_{K}}(\widetilde{\omega_{1,\PP}}\widetilde{\omega_{2,\PPP}}) = \beta_{\PPP}^{n_{\PPP}}\Lab^{\Delta_{K}}(\widetilde{\omega_{1,\PP}}\widetilde{\omega_{2,\PPP}}), \label{68}
    \end{equation}
    where $\Laa^{\Delta_{K}}, \Lab^{\Delta_{K}}$ are isotypic components of $\Laa,\Lab$ respectively with respect to the trivial character of $\Delta_{K}$.
    Using \eqref{ps2}, we rewrite \eqref{68} as
    \begin{equation}
        \alpha_{\PPP}^{n_{\PPP}}\Laa^{\Delta_{K}(\widetilde{\omega_{2,\PPP}})}(\widetilde{\omega_{1,\PP}})= \beta_{\PPP}^{n_{\PPP}}\Lab^{\Delta_{K}(\widetilde{\omega_{2,\PPP}})}(\widetilde{\omega_{1,\PP}}). \label{69}
    \end{equation}
    From Lemma \ref{lem3.1}, we know that $\Laa^{\Delta_{K}(\widetilde{\omega_{2,\PPP}})}, \Lab^{\Delta_{K}(\widetilde{\omega_{2,\PPP}})} \in \HH_{E', v_{p}(\alpha_{\PP})}(\GG_{\PP})$ for some extension $E'$ of $E$.
    As $v_{p}(\alpha_{\PP}) < k+1$, using Lemma \ref{lem3.5} we have
    \begin{equation*} \alpha_{\PPP}^{n_{\PPP}}\Laa^{\Delta_{K}(\widetilde{\omega_{2,\PPP}})}=\beta_{\PPP}^{n_{\PPP}}\Lab^{\Delta_{K}(\widetilde{\omega_{2,\PPP}})}. 
    \end{equation*}

    Hence, their partial derivatives also agree, i.e.
    \begin{equation}
        \alpha_{\PPP}^{n_{\PPP}}\delp\Laa^{\Delta_{K}(\widetilde{\omega_{2,\PPP}})} = \beta_{\PPP}^{n_{\PPP}}\delp\Lab^{\Delta_{K}(\widetilde{\omega_{2,\PPP}})}. \label{70}
    \end{equation}
    But, for any power series $F\in \HH_{E, v_{p}(\alpha_{\PP}),s}$,
    $$\delp(F^{(\widetilde{\omega_{2,\PPP}})})(\widetilde{\omega_{1,\PP}}) = (\delp F)(\widetilde{\omega_{1,\PP}}\widetilde{\omega_{2,\PPP}}),$$
    for all Hecke characters  $\omega_{1} \in S_{\PP}$.
    We have this equality since we are partially differentiating with respect to the variable $\gamma_{\PP}-1$. 
    More precisely: $F^{(\widetilde{\omega_{2,\PPP}})}\in\HH_{E', v_{p}(a_{\PP})}(\GG_\PP)$.
    Thus $\delp(F^{(\widetilde{\omega_{2,\PPP}})})(\widetilde{\omega_{1,\PP}})$ is first differentiating and then evaluating at the character $\widetilde{\omega_{1,\PP}}$. 
    On the other hand, $\delp(F)(\widetilde{\omega_{1,\PP}}\widetilde{\omega_{2,\PPP}})$ is first partially differentiating and then evaluating at the character $\widetilde{\omega_{1,\PP}}\widetilde{\omega_{2,\PPP}}$.

    Thus, for any Hecke character $\Xi$,
    \begin{equation}
        \alpha_{\PPP}^{n_{\PPP}}\delp\Laa^{\Delta_{K}}(\widecheck{\Xi})= \beta_{\PPP}^{n_{\PPP}}\delp\Lab^{\Delta_{K}}(\widecheck{\Xi}), \label{71}
    \end{equation}
    Since equation \eqref{71} is true for any isotypic component, we have
    \begin{equation}
        \alpha_{\PPP}^{n_{\PPP}}\delp\Laa(\tilX)= \beta_{\PPP}^{n_{\PPP}}\delp\Lab(\tilX). \label{72}
    \end{equation}
\end{proof}

\begin{prop}
    \label{prop3.8}
    There exist $\Lss,\Lsf,\Lfs,\Lff \in \Lambda_{E}(G_{p^{\infty}})$ such that
    \begin{equation*}
        \begin{pmatrix}
            \Lsa & \Lfa\\[6pt]
            \Lsb & \Lfb
        \end{pmatrix} =
        (Q_{\PPP}^{-1}\Mbarpp)\begin{pmatrix} \Lss & \Lfs \\[6pt] \Lsf & \Lff \end{pmatrix}.
    \end{equation*}
\end{prop}

\begin{proof}
    The proof is similar to the proof of \cite[Proposition 2.5]{Lei3}. 
    We will prove that
    $$\begin{pmatrix}
    \Lsa \\[6pt] \Lsb
    \end{pmatrix} = (Q^{-1}_{\PPP}\Mbarpp)\begin{pmatrix} \Lss \\[6pt] \Lsf\end{pmatrix}.$$
    The proof for the other set of power series is similar.

Let $\omega_{1}\in S_{\PP}$ and $\omega_{2}\in S_{\PPP}$ and $\Xi=\omega_{1}.\omega_{2}$.  
    Recall, from the proof of Proposition \ref{prop3.6}, for $*\in\{\alpha_{\PPP},\beta_{\PPP}\}$, we have
    \begin{equation}
        L_{\sharp,*}= \dfrac{P_{4,\PP}L_{\alpha_{\PP},*} - P_{2,\PP}L_{\beta_{\PP},*}}{\det(Q^{-1}_{\PP}\Mbarp)}.
    \end{equation}
Thus,
\begin{equation}
    L_{\sharp,*}\det(Q^{-1}_{\PP}\Mbarp)= P_{4,\PP}L_{\alpha_{\PP},*} - P_{2,\PP}L_{\beta_{\PP},*}.\label{74}
\end{equation}

From Theorem \ref{thm3.2}, we know that $\det(\Mbarp)$ is equal to, upto a unit in $\Lambda_{E}(\GG_{\PP})$, $\dfrac{\log_{p,k-1}(\gamma_{\PP})}{\delta_{k-1}(\gamma_{\PP}-1)}$.
Thus, the zeros of of $\det(Q^{-1}_{\PP}\Mbarp)$ are of type $u^{j}\zeta-1$, where $\zeta$ is a $p^{n}$-th root of unity, for all $0\leq j\leq k$ and $n\geq 1$.
Therefore, by the definition of $\widetilde{\omega_{1,\PP}}$, we have $\det(Q^{-1}_{\PP}\Mbarp)(\widetilde{\omega_{1,\PP}})=0$.
Furthermore, since all the zeros of $\det(Q^{-1}_{\PP}\Mbarp)$ are simple, we conclude $(\delp\det(Q^{-1}_{\PP}\Mbarp))(\widetilde{\omega_{1,\PP}})\neq 0$.

For the rest of the proof, we will use isotypic components corresponding to the trivial character of $\Delta_{K}$. 
From \eqref{74}, we get,
\begin{equation*}
\delp L^{\Delta_{K}}_{\sharp,*}\det(Q^{-1}_{\PP}\Mbarp)+L^{\Delta_{K}}_{\sharp,*}\delp\det(Q^{-1}_{\PP}\Mbarp)=\delp P_{4,\PP}L^{\Delta_{K}}_{\alpha_{\PP},*}+P_{4,\PP}\delp L^{\Delta_{K}}_{\alpha_{\PP},*}-(\delp P_{2,\PP}L^{\Delta_{K}}_{\beta_{\PP},*}+P_{2,\PP}\delp L^{\Delta_{K}}_{\beta_{\PP},*}).\label{75}       
\end{equation*}
We evaluate the above equation at $\widecheck{\Xi}=\tilX|_{\overline{\langle \gamma_{\PP}\rangle} \times \overline{\langle \gamma_{\PPP}\rangle}}$, where $\Xi=\omega_{1}\omega_{2}$ and apply Proposition \ref{prop3.7} together with the equations \eqref{44} to \eqref{47}  to get
\begin{equation}
    L^{\Delta_{K}}_{\sharp,*}(\widecheck{\Xi}).(\delp\det(Q^{-1}_{\PP}\Mbarp))(\widecheck{\Xi}) = (*)^{-n_{\PPP}}\cdot M_{\tilX},
\end{equation}
where $M_{\tilX}$ is the constant
$$(\delp P_{4,\PP})(\widecheck{\Xi})(\alpha_{\PP}^{-n_{\PP}}C_{a,b,\tilX}) + P_{4,\PP}(\widecheck{\Xi})A_{a,b,\tilX} - (\delp P_{2,\PP})(\widecheck{\Xi})(\beta_{\PP}^{-n_{\PP}}C_{a,b,\tilX}) - P_{2,\PP}(\widecheck{\Xi})B_{a,b,\tilX}.$$
In other words, we have
\begin{align*}
    \Lsa^{\Delta_{K}}(\widecheck{\Xi}) &= \alpha_{\PPP}^{-n_{\PPP}}\dfrac{M_{\tilX}}{(\delp\det(Q^{-1}_{\PP}\Mbarp))(\widecheck{\Xi})}, \\[6pt]
    \Lsb^{\Delta_{K}}(\widecheck{\Xi}) &= \beta_{\PPP}^{-n_{\PPP}}\dfrac{M_{\tilX}}{(\delp\det(Q^{-1}_{\PP}\Mbarp))(\widecheck{\Xi})}.
\end{align*}

Since $\Xi=\omega_{1}\omega_{2}$, we can rewrite the above equations as
\begin{align*}
    \Lsa^{\Delta_{K}(\widetilde{\omega_{1,\PP}})}(\widetilde{\omega_{2,\PPP}}) &= \alpha_{\PPP}^{-n_{\PPP}}\dfrac{M_{\tilX}}{(\delp\det(Q^{-1}_{\PP}\Mbarp))(\widecheck{\Xi})} , \\[6pt]
    \Lsb^{\Delta_{K}(\widetilde{\omega_{1,\PP}})}(\widetilde{\omega_{2,\PPP}}) &= \beta_{\PPP}^{-n_{\PPP}}\dfrac{M_{\tilX}}{(\delp\det(Q^{-1}_{\PP}\Mbarp)).(\widecheck{\Xi})}.
\end{align*}

Thus, after using Theorem \ref{thm1} and the proof of Proposition \ref{prop3.6} (we need to change $\PP$ with $\PPP$), we deduce the desired result.
\end{proof}

\begin{proof}[Proof of Theorem \ref{mainthm2}]
   Combining the factorizations obtained in Propositions \ref{prop3.6} and \ref{prop3.8}, we deduce the result.
\end{proof}


\appendix
\section{Signed $p$-adic $L$-functions for non-parallel weight Bianchi modular forms}
\label{appA}

In this appendix, we extend the results from parallel weight cuspidal Bianchi modular forms to non-parallel weight $C$-cuspidal Bianchi modular forms. 
The notion of $C$-cuspidality is related with the vanishing of the constant term of Fourier expansions of Bianchi modular forms (with level divisible by $p$) at suitable cusps.
For the definitions and proper explanations about $C$-\emph{cuspidality}, see \cite[Section 2]{LPS2}.
Note that the space of cuspidal Bianchi modular forms with level at $p$ is a proper subset of the set of $C$-cuspidal Bianchi modular forms with level at $p$.

Let us fix some notations first.
Fix an odd prime $p$.
Let $K/\QQ$ be a quadratic imaginary field and $p$ splits in $K$ as $p\OO_{K}=\PP\PPP$.
We also assume $p$ does not divide the class number $h_K$ of $K$.
For $\qq\in\pset$, let $K_{\qq}\cong\Qp$ be the completion of $K$ at the prime $\qq$.
Let $\OO_{K_{\qq}}$ be the ring of integers of $K_{\qqq}$ and let $\varpi_{\qq}$ be its uniformizer.
We fix the embeddings $\iota_{\infty}:\overline{\QQ}\hookrightarrow\CC$ and $\iota_{p}:\overline{\QQ}\hookrightarrow\overline{\Qp}$.
Note that $\iota_p$ fixes a $p$-adic valuation $v_p$ on $\overline{\Qp}$.
Hence, we choose $\iota_p$ such that $v_{p}(\varpi_{\PP})=1$ and $v_{p}(\varpi_{\PPP})=0$.
The embeddings $\iota_{\infty}$ and $\iota_p$ will give the isomorphism $\iota:\CC\xrightarrow{\cong}\overline{\Qp}$ satisfying $\iota\circ\iota_{\infty}=\iota_p$.
Fix non-negative integers $k$ and $\ell$.
\subsection{$p$-adic $L$-functions associated to $C$-cuspidal Bianchi modular forms}
Let $\FF$ be a $C$-cuspidal Bianchi eigenform of weight $(k,\ell)$ and level $\frkn$, where $p$ divides $\frkn$. 
For the Fourier expansion related to any Bianchi modular form, see \cite[Section 2.3]{LPS2}.
For $\qq\in\pset$, let $a_{\qq}$ be the $U_\qq$-eigenvalue of $\FF$. Moreover, $\FF$ is \emph{small slope}, i.e., $v_{p}(a_{\PP})< k+1$ and $v_{p}(a_{\PPP})<\ell+1$.

Let $c(\cdot,\FF)$ denote the Fourier coefficients of $\FF$. Like in the cuspidal case, for any Hecke character $\Xi$ with the conductor $\mathfrak{f}$, we define the \emph{$L$-function of $\FF$ twisted by $\Xi$}:
\[L(\FF,\Xi,s)=\sum_{\substack{0\neq\mathfrak{a}\subset\OO_F,\\ (\mathfrak{a},\mathfrak{f})=1}} c(\mathfrak{a}, \FF)\Xi(\mathfrak{a})\mathrm{N}(\mathfrak{a})^{-s},\]
where $s\in\CC$.

Using Deligne's $\GG$-factors, we \emph{renormalize} this $L$-function:
\[\Lambda(\FF,\Xi,s)=\dfrac{\Gamma(q+s)\GG(r+s)}{(2\pi i)^{q+s}(2\pi i)^{r+s}}L(\FF,\Xi,s),\] where $(q,r)$ is the infinity type of $\Xi$.

The main theorem of \cite{LPS2} is:
\begin{thm}{\cite[Theorem 4.12]{LPS2}}
\label{thmA1}
For chosen embeddings $\iota_{\infty},\iota_p$, and $\iota$, there exists a locally analytic distribution $L^{\iota}_{p,\FF}$ on the ray class group $G_{p^\infty}$ such that for any Hecke character $\Xi$ of $K$ of conductor $\PP^{n_{\PP}}\PPP^{n_{\PPP}}$ and infinity type $(0,0)\leq (q,r) \leq (k,\ell),$ we have
\begin{equation}
    \label{padccusp}
    L^{\iota}_{p,\FF}(\tilX)=(\text{explicit factor})\times\dfrac{1}{\lambda_{\mathfrak{f}}}\times\Lambda(\FF,\Xi, 1),
\end{equation}
where $\lambda_{\mathfrak{f}}$ is the $U_{\mathfrak{f}}=\prod_{\qq\mid p} U_{\qq}^{n_{\qq}}$-eigenvalue of $\FF$.\\The distribution $L^{\iota}_{p,\FF}$ is $(v_{p}(a_{\qq}))_{\qq\mid p}$-admissible and therefore is unique.
\end{thm}
From the Remark \ref{rem7.7}, we can see $L^{\iota}_{p,\FF}\in \HH_{E, v_{p}(a_{\PP}), v_{p}(a_{\PPP})}(G_{p^{\infty}}),$ for some finite extension $E/\Qp$.

\subsection{Decomposition of $p$-adic $L$-functions of $C$-cuspidal Bianchi modular form}
We first fix a Bianchi modular eigenform $\FF$ of weight $(k,\ell)$ and level $\frkm$, where $\frkm$ is coprime with $p$.
Furthermore, we assume that $\FF$ vanishes at cusps $0$ and $\infty$.

Consider the $L$-function of $\FF$:
\[L(\FF,s)=\sum_{0\neq \mathfrak{a}\subset\OO_K} c(\mathfrak{a},\FF)\mathrm{N}(\mathfrak{a})^{-s},\]
then the \emph{local Euler factor} at $\qq\in\pset$ is
\begin{equation}
    \label{localfact}
    L_{\qq}(\FF,s)^{-1}= 1- a_{\qq}p^{-s} +\varepsilon_{\FF}(\varpi_{\qq})p^{1-2s},
\end{equation}
where $a_{\qq}$ is the $T_{\qq}$-eigenvalue of $\FF$, and $\varepsilon_{\FF}$ is the central character associated to $\FF$. Recall that $\varepsilon_{\FF}$ is a Hecke character of the infinity type $(-k,-\ell)$, and conductor coprime with $p$.

Hence, we will consider the following Hecke polynomials: for the prime ideal $\PP$, we have
\begin{align}
    P_{\PP}(X)&\coloneqq X^{2}-a_{\PP}X + \varepsilon(\varpi_{\PP})p,\\ &= X^{2}-a_{\PP}X + \varepsilon_{\infty}(\varpi_{\PP})^{-1}p,  \\
    &= X^{2}-a_{\PP}X + \varpi_{p}^{k}\cdot\varpi_{\PP}^{\ell}\cdot p,\\
    &= X^{2}-a_{\PP}X + \varpi_{\PP}^{k+1}\varpi_{\PPP}^{\ell+1}.
\end{align}
Similarly, for prime $\PPP$, we consider
\begin{equation}
    P_{\PPP}(X)\coloneqq X^{2}-a_{\PPP}X +\varpi_{\PPP}^{k+1}\varpi_{\PP}^{\ell+1}.
\end{equation}

From now onwards, we assume $\FF$ is non-ordinary at both the primes $\PP$ and $\PPP$, i.e., $v_{p}(a_{\PP}), v_{p}(a_{\PPP})>0$. 
We furthermore assume 
\begin{enumerate}
    \item $v_{p}(a_{\PP}) > \left\lfloor \dfrac{k}{p-1}\right\rfloor$;\vspace{2.5mm}
    \item $v_{p}(a_{\PPP}) > \left\lfloor \dfrac{\ell}{p-1} \right\rfloor$.
\end{enumerate}

For $\qq\in\pset$, let $\alpha_{\qq}$ and $\beta_{\qq}$ be the roots of $P_{\qq}(X)$.
Recall that we have a $p$-adic valuation $v_p$ corresponding to the fixed embedding $\iota_p$ such that $v_{p}(\varpi_{\PP})=1$ and $v_{p}(\varpi_{\PPP})=0$. 
Hence, the roots $\alpha_{\PP}$ and $\beta_{\PP}$ satisfy:
\begin{align*}
    v_{p}(\alpha_{\PP}\beta_{\PP}) &=v_{p}(\varpi_{\PP}^{k+1}\varpi_{\PPP}^{\ell+1}),\\
    v_{p}(\alpha_{\PP})+v_{p}(\beta_{\PP})&=v_{p}(\varpi_{\PP}^{k+1})+v_{p}(\varpi_{\PPP}^{\ell+1}),\\
    &= k+1,
\end{align*}
and therefore $$0 < v_{p}(\alpha_{\PP}), v_{p}(\beta_{\PP})< k+1.$$
Similarly,for the roots $\alpha_{\PPP}$ and $\beta_{\PPP}$, we have 
\[v_{p}(\alpha_{\PPP})+ v_{p}(\beta_{\PPP}) =\ell+1,\]
and 
\[0 < v_{p}(\alpha_{\PPP}),  v_{p}(\beta_{\PPP}) < \ell+1. \] 
We assume $\alpha_{\qq}\neq\beta_{\qq}$ for $\qq\in\pset$.

Like in the cuspidal case, we have four $p$-stabilizations of $\FF: \FF^{\alpha_{\PP},\alpha_{\PPP}}, \FF^{\alpha_{\PP},\beta_{\PPP}}, \FF^{\beta_{\PP},\alpha_{\PPP}}, \FF^{\beta_{\PP},\beta_{\PPP}}$.
Note that we have assumed $\FF$ vanishes at the cusps $0$ and $\infty$.

\begin{lem}
    For $*\in\{\alpha_{\PP},\beta_{\PP}\}$ and $\dagger\in\{\alpha_{\PP},\beta_{\PP}\}$, the $p$-stabilization $\FF^{*,\dagger}$ is a $C$-cuspidal Bianchi modular form of weight $(k,\ell)$ and level $p\frkm$.
\end{lem}
\begin{proof}
    The proof is similar to the proof of \cite[Proposition 5.3]{LPS2}.
\end{proof}

Therefore, for $*\in\{\alpha_{\PP}, \beta_{\PP}\}$ and $\dagger\in\{\alpha_{\PPP}, \beta_{\PPP}\}$, the $p$-stabilizations $\FF^{*,\dagger}$ are $C$-cuspidal, of level $p\frkm$, and are of \emph{small slope}, since $v_{p}(*)< k+1$ and $v_{p}(\dagger)< \ell+1$. 
Hence, by Theorem \ref{thmA1}, we can attach a $p$-adic $L$-function $L^{\iota}_{*,\dagger}\coloneqq L^{\iota}_{p,\FF^{*,\dagger}}\in\HH_{E, v_{p}(*), v_{p}(\dagger)}(G_{p^\infty})$ to the $C$-cuspidal Bianchi modular form $\FF^{*,\dagger}$. 
Furthermore, for any Hecke character $\Xi$ of the infinity type $(q,r)$ such that $0\leq q \leq k$ and $0 \leq r\leq \ell$ of conductor $\PP^{n_{\PP}}\PPP^{n_{\PPP}}$, where $n_{\PP}, n_{\PPP}\in \ZZ_{>0}$, we have the following interpolation properties:
\begin{align*}
    L^{\iota}_{\alpha_{\PP},\alpha_{\PPP}}(\tilX)&=\alpha_{\PP}^{-n_{\PP}}\alpha_{\PPP}^{-n_{\PPP}}\cdot C_{q,r,\tilX},\\
    L^{\iota}_{\alpha_{\PP},\beta_{\PPP}}(\tilX)&=\alpha_{\PP}^{-n_{\PP}}\beta_{\PPP}^{-n_{\PPP}}\cdot C_{q,r,\tilX},\\
    L^{\iota}_{\beta_{\PP},\alpha_{\PPP}}(\tilX)&=\beta_{\PP}^{-n_{\PP}}\alpha_{\PPP}^{-n_{\PPP}}\cdot C_{q,r,\tilX},\\
    L^{\iota}_{\beta_{\PP},\beta_{\PPP}}(\tilX)&=\beta_{\PP}^{-n_{\PP}}\beta_{\PPP}^{-n_{\PPP}}\cdot C_{q,r,\tilX},
\end{align*}
where $C_{q,r,\tilX}\in\overline{\Qp}$ is a constant independent of $\alpha_{\PP},\beta_{\PP}, \alpha_{\PPP},\beta_{\PPP}$. 
More precisely, $C_{q,r,\tilX}=\text{(some explicit factor)}\times\Lambda(\FF, \Xi, 1)$, since the conductor of the Hecke character $\Xi$ is $\PP^{n_{\PP}}\PPP^{n_{\PPP}}$ with $n_{\PP}, n_{\PPP}\in\ZZ_{>0}$, and hence $L(\FF,\Xi, 1)=L(\FF^{*,\dagger}, \Xi,1)$ for all $*\in\{\alpha_{\PP},\beta_{\PP}\}$ and $\dagger\in\{\alpha_{\PPP},\beta_{\PPP}\}$.

Note that, since $v_{p}(\varpi_{\PP})=1$ and $v_{p}(\varpi_{\PPP})=0$, we can write $\varpi_{\PP}^{k+1}\varpi^{\ell+1}=p^{k+1}\cdot u_{\PP},$ where $u_\PP$ is a suitable unit in some ring of integers of finite extension of $\Qp$. 
Similarly, we can write $\varpi_{\PP}^{\ell+1}\varpi_{\PPP}^{k+1}=p^{\ell+1}\cdot v_{\PP}$.
Let $E/\Qp$ be a finite extension of $\Qp$ large enough to contain all Hecke eigenvalues of $\FF$, $\alpha_{\PP},\beta_{\PP},\alpha_{\PPP},\beta_{\PPP}, u^{1/2}_{\PP},$ and $v^{1/2}_{\PP}$.

We construct logarithmic matrices using the methods from Sections \ref{sec3},\ref{sec4}, and \ref{sec5}.

For the prime $\PP$, let 
\begin{align*}
    A_{\vp,\PP}&=\begin{pmatrix}
        0 & \dfrac{-1}{p^{k+1}u_{\PP}} \\[12pt]  1 & \dfrac{a_{\PP}}{p^{k+1}u_\PP},
    \end{pmatrix},\\[10pt]
    Q_{\PP}&= \begin{pmatrix} \alpha_{\PP} & -\beta_{\PP} \\[9pt]
-p^{k+1}u_{\PP} & p^{k+1}u_{\PP} \end{pmatrix}.
\end{align*}
Using this data, we construct a logarithmic matrix $\underline{M^{k}(\PP)}\in M_{2,2}(\HH_{E}(\GG_1))$ such that it satisfies:
\begin{enumerate}
    \item If $Q_{\PP}^{-1}(\underline{M^{k}(\PP)})=\begin{pmatrix}
            P_{1}(\PP) & P_{2}(\PP) \\
            P_{3}(\PP) & P_{4}(\PP)
        \end{pmatrix},$
        then $P_{1}(\PP), P_{2}(\PP)\in \HH_{E,v_{p}(\alpha_{\PP})}(\GG_1)$ and $P_{3}(\PP), P_{4}(\PP)\in \HH_{E,v_{p}(\beta_{\PP})}(\GG_1)$.
        \item The second row of $A_{\varphi, \PP}^{-n}\underline{M^{k}(\PP)}$ is divisible by $\Phi_{n-1,k+1}(\gamma_{0})$ over $\HH_{E}(\GG_1)$.
        \item The determinant $\det(\underline{M^{k}(\PP)})$ is $\dfrac{\log_{p,k+1}(\gamma_{0})}{\delta_{k+1}(\gamma_{0}-1)}$, up to a unit in $\Lambda_{E}(\GG_1)$. 
    \end{enumerate}
Define $\underline{M^{k}_{\PP}}\coloneqq\mat_{\PP}(\underline{M^{k}(\PP)}) \in M_{2,2}(\HH_{E}(\GG_\PP))$.

For the prime $\PPP$, let
\begin{align*}
    A_{\vp,\PPP}&=\begin{pmatrix}
        0 & \dfrac{-1}{p^{\ell+1}v_{\PP}} \\[12pt]  1 & \dfrac{a_{\PPP}}{p^{\ell+1}v_{\PP}},
    \end{pmatrix},\\[10pt]
    Q_{\PPP}&= \begin{pmatrix} \alpha_{\PPP} & -\beta_{\PPP} \\[9pt]
-p^{\ell+1}v_{\PP} & p^{\ell+1}v_{\PP} \end{pmatrix}.
\end{align*}
Using this, we can construct a logarithmic matrix $\underline{M^{\ell}(\PPP)}\in\HH_{E}(\GG_1)$ such that:
\begin{enumerate}
    \item If $Q_{\PPP}^{-1}(\underline{M^{\ell}(\PPP)})=\begin{pmatrix}
            P_{1}(\PPP) & P_{2}(\PPP) \\
            P_{3}(\PPP) & P_{4}(\PPP)
        \end{pmatrix},$
        then $P_{1}(\PPP), P_{2}(\PPP)\in \HH_{E,v_{p}(\alpha_{\PPP})}(\GG_1)$ and $P_{3}(\PPP), P_{4}(\PPP)\in \HH_{E,v_{p}(\beta_{\PPP})}(\GG_1)$.
        \item The second row of $A_{\varphi, \PPP}^{-n}\underline{M^{\ell}(\PPP)}$ is divisible by $\Phi_{n-1,\ell+1}(\gamma_{0})$ over $\HH_{E}(\GG_1)$.
        \item The determinant $\det(\underline{M^{\ell}(\PP)})$ is $\dfrac{\log_{p,\ell-1}(\gamma_{0})}{\delta_{\ell+1}(\gamma_{0}-1)}$, up to a unit in $\Lambda_{E}(\GG_1)$. 
\end{enumerate}
Define $\underline{M^{\ell}_{\PPP}}\coloneqq\mat_{\PPP}(\underline{M^{\ell}(\PPP)})\in\HH_{E}(\GG_{\PPP})$.

Note that we are getting $k+1,\ell+1$ instead of $k-1, \ell-1$ since $k,\ell\in \ZZ_{\geq 0}$.

Hence, by following the same methods used in the proof of Propositions \ref{prop3.6}, \ref{prop3.7}, \ref{prop3.8}, we can conclude:
\begin{thm}
    \label{mainthm3}
There exist $L^{\iota}_{\sharp,\sharp}, L^{\iota}_{\sharp,\flat}, L^{\iota}_{\flat,\sharp}, L^{\iota}_{\flat,\flat}\in \Lambda_{E}(G_{p^\infty})$ such that
\begin{equation}
    \begin{pmatrix}
            L^{\iota}_{\alpha_{\PP},\alpha_{\PPP}} & L^{\iota}_{\beta_{\PP},\alpha_{\PPP}}\\[6pt]
            L^{\iota}_{\alpha_{\PP},\beta_{\PPP}} & L^{\iota}_{\beta_{\PP},\beta_{\PPP}}
        \end{pmatrix}
        = Q^{-1}_{\PPP}\underline{M^{\ell}_{\PPP}} \begin{pmatrix}
            L^{\iota}_{\sharp,\sharp} & L^{\iota}_{\flat,\sharp}\\[6pt]
            L^{\iota}_{\sharp,\flat} & L^{\iota}_{\flat,\flat}
        \end{pmatrix} (Q^{-1}_{\PP}\underline{M^{k}_{\PP}})^{T}.
\end{equation}
\end{thm}

\subsection*{Data Availability}
Data sharing is not applicable to this article.

\subsection*{Conflict of interests}
There are no conflict of interests. 

\medskip

\nocite{*} 
\printbibliography[title={References}]

\end{document}